\documentclass[final]{amsart}
\usepackage{a4}
\usepackage{amsmath}%
\usepackage{amstext}%
\usepackage{amssymb}%
\usepackage{amsthm}
\usepackage{comment}
\usepackage{listings}
\usepackage{url}

\usepackage[final]{graphicx}
\usepackage{subcaption}

\usepackage[margin=0.5in]{geometry}

\usepackage{graphicx}
\usepackage{epstopdf}

\usepackage{xcolor}

\newtheorem{theorem}{Theorem}

\newtheorem{corollary}[theorem]{Corollary}

\newtheorem{lemma}[theorem]{Lemma}

\newtheorem{rem}[theorem]{Remark}
\newenvironment{remark}{\rem\rm}{\endrem}

\newcounter{unnumber}

\newtheorem{prf}[unnumber]{Proof}
\renewenvironment{proof}{\prf\rm}{\hfill{$\blacksquare$}\endprf}
\newcommand{\R}{\mathbb{R}}%
\newcommand{\N}{\mathbb{N}}%
\newcommand{\e}{\varepsilon}%
\newcommand{\ol}{\overline}%

\newcommand{\n}{{\nabla}}

\newcommand{\To}{\longrightarrow}
\def\a{\alpha}

\def\b{\beta}
\def\e{\epsilon}
\def\d{\delta}

\def\g{\gamma}

\def\<{\langle}
\def\>{\rangle}
\DeclareMathOperator*\proj{proj}%
\DeclareMathOperator*\argmin{argmin}

\author{Szil\'ard Csaba L\'aszl\'o  }
\thanks{Corresponding Author: Szil\'ard Csaba L\'aszl\'o, Technical University of Cluj-Napoca, Department of Mathematics, Str. Memorandumului nr. 28, 400114 Cluj-Napoca, Romania, e-mail: szilard.laszlo@math.utcluj.ro}
\thanks{This work was supported by a grant of the Ministry of Research, Innovation and Digitization, CNCS -
UEFISCDI, project number PN-III-P1-1.1-TE-2021-0138, within PNCDI III}

\title{Solving convex optimization problems via a second order dynamical system with implicit Hessian damping and Tikhonov regularization}

\begin{document}

\begin{abstract}
This paper deals with a second order dynamical system with  a Tikhonov regularization term in connection  to the minimization problem of a convex Fr\'echet differentiable function.  The  fact that beside the asymptotically vanishing damping we also consider  an implicit Hessian driven damping  in the dynamical system under study allows us, via straightforward  explicit discretization, to obtain inertial algorithms of gradient type.
We show that the  value of the objective function in a generated trajectory converges rapidly to the global minimum of the objective function and depending the Tikhonov regularization parameter  the generated trajectory converges weakly to a minimizer of the objective function or the generated trajectory converges strongly to the element of minimal norm from the $\argmin$ set of the objective function. We also  obtain the fast convergence of the velocities towards zero and some integral estimates. Our analysis reveals that the Tikhonov regularization parameter and the damping parameters are strongly correlated, there is a setting of the parameters that separates the cases when  weak convergence of the trajectories to a minimizer and strong convergence of the trajectories to the minimal norm minimizer can be obtained.
\end{abstract}

\maketitle
\vspace{1ex}


\markboth{S.C. L\'ASZL\'O}{Solving convex optimization problems via a second order dynamical system}

\section{Introduction}
\subsection{The problem formulation and a model result}

In this paper we study the asymptotical behaviour of the trajectories generated by the second order dynamical system
 \begin{align}\label{DynSys}\tag{DS}
\begin{cases}
\ddot{x}(t) + \frac{\a}{t^q} \dot{x}(t) + \nabla g\left(x(t) +\left(\g+\frac{\b}{t^q}\right)\dot{x}(t)\right) +\e(t)x(t)=0,\\
x(t_0) = u_0\in\mathcal{H}, \,
\dot{x}(t_0) = v_0\in\mathcal{H},
\end{cases}
\end{align}
in connection to the minimization problem
\begin{align}\label{P}\tag{P}
\inf_{x\in\mathcal{H}}g(x),
\end{align}
where $\mathcal{H}$ is a real Hilbert space endowed with the scalar product $\< \cdot,\cdot\>$ and norm $\|\cdot\|$ and
$g:\mathcal{H}\To\R$ is a convex Fr\'echet differentiable function. Throughout the paper we assume that $\n g$ is Lipschitz continuous on bounded sets and  the solution set of $g$, namely $\argmin g,$ is non-empty.

The starting time in \eqref{DynSys} is $t_0 > 0$ and we assume that $\a>0,\,q\in]0,1[$ and  $\g > 0, \, \b\in\R$ or $\g=0,\,\b> 0$. In case $\b<0$ we also assume that $t_0>\left|\frac{\b}{\g}\right|^{\frac{1}{q}}.$ Further, $\e:[t_0,+\infty[\To\R_+$ is a non-increasing function of class $C^1$, such that  $\lim\limits_{t\to+\infty}\e(t)=0.$

 Note that $\e(t)x(t)$ in \eqref{DynSys} is a Tikhonov regularization term that allows us to obtain the strong convergence of a  trajectory generated by \eqref{DynSys} to the minimal norm minimizer of $g,$ see \cite{AL-nemkoz,ABCR,L-jde,BCL-jma}. Further, we do not allow $q=1$ since in that case \eqref{DynSys} becomes the dynamical system studied in \cite{ALSiopt}. However, we show that by considering $q\in]0,1[$ the results obtained in \cite{ALSiopt} can be improved. Further we avoid the case $\g=\b=0$ since then \eqref{DynSys} becomes the dynamical system studied in \cite{L-jde}. Our analysis shows that the damping parameters and the Tikhonov regularization parameter are strongly correlated, there is a setting of the parameters which assure weak convergence of the trajectories generated by \eqref{DynSys} to a minimizer of $g$ and rapid convergence of the function values to the global minimum of $g$, meanwhile for another constellation of the parameters involved, rapid convergence of the function  values and strong convergence of the generated trajectories to the minimal norm minimizer is obtained. In order to visualize the facts emphasized above, let us specify $\e(t)=\frac{a}{t^p},\,a,p>0$ and let us denote $\b(t)=\g+\frac{\b}{t^q}.$ It is easy to see that $\lim\limits_{t \to +\infty} \e(t) = 0$, $\e(t)$ is a non-increasing function of class $C^1$ and satisfies $\e(t) \ge 0$ for every $t \ge t_0 > 0$. In this case the following result holds, see Corollary \ref{specweak} and Corollary \ref{specstr} and Theorem \ref{pq1}.
\begin{theorem}\label{ModelResult}
 Suppose that the regularization parameter has the form $\e(t)=\frac{a}{t^p},\,a,p>0$ and for some starting points $u_0, \, v_0\in\mathcal{H}$ let $x : [t_0, +\infty[ \to \mathcal{H}$ be the unique global solution of  \eqref{DynSys}. Then, the following results hold.
\begin{itemize}
\item[$(i)$] If $q+1<p\le 2$ and for $p=2$ one has  $a\ge q(1-q),$ then  the trajectory $x(t)$ converges weakly, as $t\to+\infty$, to an element of $\argmin g$. Further the following pointwise and integral estimates are valid.\\
One has $\| \dot{x}(t) \|=o\left(t^{-q}\right),$
$ g\left(x(t) + \b(t)\dot{x}(t)\right) - \min g =o\left(t^{-2q}\right) \text{ as } t\to+\infty$ and $\int_{t_0}^{+\infty} t^q \|\dot{x}(t)\|^2dt<+\infty,\, \int_{t_0}^{+\infty} t^q \left( g\left(x(t) + \b(t)\dot{x}(t)\right) - \min g\right)dt<+\infty.$ Concerning the gradient one has the integral estimates
$\int_{t_0}^{+\infty} t^{2q}\left\|\n g\left(x(t) + \b(t)\dot{x}(t)\right)\right\|^2dt<+\infty,$ if $\g>0,\,\b\in\R$
and
$ \int_{t_0}^{+\infty} t^{2q-1} \left\|\n g\left(x(t) + \b(t)\dot{x}(t)\right)\right\|^2dt<+\infty,$ if $\g=0,\,\b>0.$
\item[$(ii)$] If $q\le p<q+1$ and when $q=p$ one has $a\le \dfrac{\a}{2\g},$ then the trajectory $x(t)$ converges strongly, as $t \to+\infty$, to $x^*$, the element of minimum norm of $\argmin g$. Further the following pointwise and integral estimates are valid.\\
One has $\|\dot{x}(t) \|=\mathcal{O}(t^{-\frac{\min(2q,p)}{2}})$
and $g(x(t) + \b(t)\dot{x}(t) )-\min g=\mathcal{O}(t^{-\min(2q,p)})\mbox{ as }t\to+\infty,$ and if $p>2q$ then
$\int_{t_0}^{+\infty} t^q\|\dot{x}(t)\|^2 dt<+\infty,\,\int_{t_0}^{+\infty} t^q \left( g\left(x(t) + \b(t)\dot{x}(t)\right) - \min g\right)dt<+\infty.$ Concerning the gradient one has the integral estimates
$\int_{t_0}^{+\infty}t^{2q-1}\| \n g(x(t) + \b(t) \dot{x}(t))\|^2dt<+\infty,\mbox{ if }\g=0,\,\b>0\mbox{ and }p>2q$
and
$\int_{t_0}^{+\infty}t^{2q}\| \n g(x(t) + \b(t) \dot{x}(t))\|^2dt<+\infty$ if $\g>0,\,\b\in\R$ and $p>\max\left(\frac{2q+1}{2},2q\right).$
\item[$(iii)$] The case $p=q+1$ is critical in the sense that neither weak nor strong convergence of the trajectories can be obtained.  However, also in this case the trajectory $x(t)$ is bounded and one has $\| \dot{x}(t) \|=\mathcal{O}\left(t^{-q}\right),$
$ g\left(x(t) + \b(t)\dot{x}(t)\right) - \min g =\mathcal{O}\left(t^{-2q}\right) \text{ as } t\to+\infty$

\end{itemize}
\end{theorem}

\subsection{Motivation and related works}
Starting with the seminal work of Su, Boyd and Cand\`es \cite{SBC}, second order dynamical systems have been intensively studied in the literature in connection to the minimization problem having in its objective a (convex) smooth function, see \cite{ALP,att-c-p-r-math-pr2018,alv-att-bolte-red,AC,ACFR,AGR,att-p-r-jde2016,APR,BM,BCL-jee,BCL-AA,CAG1,CAG2,MJ,SDJS}. These dynamical systems lead via implicit/explicit discretizations to inertial algorithms, hence the asymptotical behaviour of the  trajectories generated by these dynamical systems give an insight into the behaviour of the sequences generated by the algorithms obtained via discretization see \cite{ALV,Nest,nesterov83,poly,L-mapr}.

Indeed, according to \cite{SBC} the dynamical system
\begin{equation}\label{ee11}
\ddot{x}(t)+\frac{\a}{t}\dot{x}(t)+\n g(x(t))=0,\,\,x(t_0)=u_0,\,\dot{x}(t_0)=v_0,\,t_0>0,\,u_0,v_0\in\mathcal{H}
\end{equation}
with $\a=3$ is the exact limit of Nesterov's accelerated convex gradient method \cite{nesterov83} and that is the reason why
\eqref{ee11} has been intensively studied in the literature in connection to the minimization problem
$\inf_{x\in\mathcal{H}}g(x).$  In \cite{SBC} the authors proved that $g(x(t))-\min g=\mathcal{O}\left(\frac{1}{t^2}\right)$
for every $\a\ge 3$ as $t\To+\infty$, however they did not show the convergence of a generated trajectory to a minimum of the objective function $g.$

In  \cite{att-c-p-r-math-pr2018}, Attouch, Chbani,  Peypouquet and Redont considered the case $\a>3$ in \eqref{ee11}, and showed that  the  generated  trajectory $x(t)$  converges weakly to  a  minimizer  of $g$ as $t\To+\infty$ and also that $g(x(t))-\min g=o\left(\frac{1}{t^2}\right)$ as $t\To+\infty$. A perturbed version of \eqref{ee11} was studied by Attouch,   Peypouquet and Redont in \cite{att-p-r-jde2016}, see also \cite{alv-att-bolte-red, ACFR, APR, SDJS}. They assumed that the objective $g$ is twice continuously differentiable and the perturbation of their system is made at the damping term. More precisely, they studied the dynamical system with Hessian driven damping
\begin{equation}\label{ee13}
\ddot{x}(t)+\frac{\a}{t}\dot{x}(t)+\b\n^2g(x(t))\dot{x}(t)+\n g(x(t))=0,\,\,x(t_0)=u_0,\,\dot{x}(t_0)=v_0,\,t_0>0,\,u_0,v_0\in\mathcal{H},
\end{equation}
where $\a>0$ and $\b>0.$
In case $\a>3,\,\b>0$ they showed the weak convergence of a generated trajectory to a minimizer of $g$. Moreover, they obtained that in this case the convergence rate of the objective function along the trajectory, that is $g(x(t))-\min g,$ is of order ${o}\left(\frac{1}{t^2}\right)$.

What one can notice concerning the dynamical systems \eqref{ee11} and \eqref{ee13} is that  these systems will never give through the natural implicit/explicit discretization Nesterov type inertial algorithms, since the gradient of $g$ is evaluated in $x(t),$ and this term via implicit/explicit discretization will become $\n g(x_n)$ or $\n g(x_{n+1})$ and never of the form $\n g(x_n+\a_n(x_n-x_{n-1}))$ as Nesterov's gradient method requires. That was the reason why in \cite{ALP}, (see also \cite{MJ}), Alecsa, L\'aszl\'o and Pin\c ta  introduced  the dynamical system
\begin{equation}\label{ee14}
\ddot{x}(t)+\frac{\a}{t}\dot{x}(t)+\nabla g\left(x(t)+\left(\g+\frac{\b}{t}\right)\dot{x}(t)\right)=0,\,
x(t_0)=u_0,\,\dot{x}(t_0)=v_0,\,t_0>0,\,u_0,v_0\in \mathcal{H}
\end{equation}
where $\a>0,\,\b\in\R,\,\g\ge 0.$  The term $\nabla g\left(x(t)+\left(\g+\frac{\b}{t}\right)\dot{x}(t)\right)$ in \eqref{ee14} is called implicit Hessian damping due to the fact that by using the Taylor expansion for $\n g(\cdot)$ we get
$
\n g\left(x(t)+\left(\g+\frac{\b}{t}\right)\dot{x}(t)\right)\approx \n g(x(t))+\left(\g+\frac{\b}{t}\right)\n^2 g(x(t))\dot{x}(t).
$
This fact also explains why the trajectories of \eqref{ee13} and \eqref{ee14} share a similar behaviour. Indeed, the objective function value in a trajectory generated by \eqref{ee14} converges in order $\mathcal{O}\left(\frac{1}{t^2}\right)$ to the global minimum of the objective function and the generated trajectory converges weakly to a minimizer of the objective function $g,$ see \cite{ALP}. Further, according to \cite{ALSiopt}, the dynamical system \eqref{ee14} can be thought as an intermediate system between the dynamical system \eqref{ee11} and the dynamical system \eqref{ee13}, which inherits the best properties of the latter systems. Note that explicit Euler discretization of \eqref{ee14} leads to Nesterov type inertial algorithms, where the gradient is evaluated at an extrapolated term of the form $x_n+\frac{\g n+\b}{n+\a}(x_n-x_{n-1}),$ see \cite{ALP}.

Tikhonov regularized versions of \eqref{ee11} and \eqref{ee13} were studied in \cite{ACR} and \cite{BCL}. In these papers, for a constellation of the parameters the authors obtained fast convergence  for the function values in a generated trajectory and weak convergence of the generated trajectory to a minimizer of the objective function, meanwhile for another setting of the parameters the strong convergence result $\liminf_{t\to+\infty}\|x(t)-x^*\|=0$ have been obtained, where $x(t)$ is the generated trajectory and $x^*$ is the minimal norm minimizer of the objective function. For a while it was an open question whether one can obtain fast convergence of the function values  and strong convergence results in a generated trajectory for the same constellation of the parameters involved. The first breakthrough has been obtained in \cite{AL-nemkoz}  where the authors   introduced the dynamical system
\begin{align}\label{DynSysABCR}
&\ddot{x}(t) + \d\sqrt{\e(t)} \dot{x}(t) +\nabla g\left(x(t)\right) +\e(t)x(t)=0,\, x(t_0) = u_0, \,
\dot{x}(t_0) = v_0,
\end{align}
with $\d>0$ and succeeded to obtain both fast convergence towards the minimal value of the objective function and the strong convergence result $\liminf_{t\to+\infty}\|x(t)-x^*\|=0$. Similar results have been obtained in \cite{ALSiopt} for the Tikhonov regularized second order dynamical system \eqref{ee14}, where the Tikhonov regularization parameter has the form $\e(t)=\frac{a}{t^2}$  and $\a>3.$ In contrast, in this paper we assume that $0<q<1$ instead of $q=1$ considered in \cite{ALSiopt} and beside fast convergence rates we are able to obtain the strong convergence result $\lim_{t\to+\infty}\|x(t)-x^*\|=0$. Moreover these results hold for every $\a>0$ and every Tikhonov regularization parameter $\e(t)=\frac{a}{t^p},\,q\le p<q+1.$

The results from \cite{AL-nemkoz} have been improved in \cite{ABCR} where the authors studied the dynamical system \eqref{DynSysABCR} and succeeded to obtain the strong convergence result $\lim_{t\to+\infty}\|x(t)-x^*\|=0.$ These results then have been extended and further improved in \cite{L-jde}, where it is also shown that the best choice for the damping parameter is not $\sqrt{\e(t)}.$ We emphasize that the results of this paper are in concordance to the results obtained in \cite{L-jde}. Further, the results from \cite{BCL} have been improved in \cite{ABCRamop} where the authors succeeded to obtain fast convergence rates and the strong convergence result $\lim_{t\to+\infty}\|x(t)-x^*\|=0.$ Since, as we emphasized before, \eqref{DynSys} and the Tikhonov regularized system with Hessian driven damping considered in \cite{ABCRamop} are strongly related, let us compare the results from this paper and the results from \cite{ABCRamop}. The second order dynamical system with Hessian driven damping and Tikhonov regularization studied in \cite{ABCRamop} is
\begin{align}\label{DynSysABCRamop}
&\ddot{x}(t) + \d\sqrt{\e(t)} \dot{x}(t)+\b\nabla^2 g\left(x(t)\right)\dot{x}(t) +\nabla g\left(x(t)\right) +\e(t)x(t)=0,\, x(t_0) = u_0, \,
\dot{x}(t_0) = v_0,
\end{align}
where $g$ is a convex function of class $C^2$, $\d,\b>0$ and $\e(t)$ is a Tikhonov regularization parameter. According to the model result from \cite{ABCRamop}, when one specify $\e(t)=\frac{1}{t^{2q}},\,0<q<1$ then the  trajectories generated by the dynamical system \eqref{DynSysABCRamop} converge strongly to the minimal norm minimizer of $g$, $g(x(t))-\min g=\mathcal{O}(t^{-2q})\mbox{ as }t\to+\infty$ and $\int_{t_0}^{+\infty}t^{3q-1}\| \n g(x(t))\|^2dt<+\infty.$ In contrast, we assume that our objective function is of class $C^1$ and  according to Theorem \ref{ModelResult} the  trajectories generated by the dynamical system \eqref{DynSysABCRamop} converge strongly to the minimal norm minimizer of $g$ whenever $q\le p<q+1$, and if $2q<p<q+1$ then
$g(x(t) + \b(t)\dot{x}(t) )-\min g=\mathcal{O}(t^{-2q})\mbox{ as }t\to+\infty,$ but we also have the pointwise estimate for the velocity $\|\dot{x}(t) \|=\mathcal{O}(t^{-q})$. Moreover, we have the integral estimates $\int_{t_0}^{+\infty} t^q\|\dot{x}(t)\|^2 dt<+\infty,\,\int_{t_0}^{+\infty} t^q \left( g\left(x(t) + \b(t)\dot{x}(t)\right) - \min g\right)dt<+\infty$ and 
$\int_{t_0}^{+\infty}t^{2q}\| \n g(x(t) + \b(t) \dot{x}(t))\|^2dt<+\infty,\mbox{ if }\g>0,\,\b\in\R \mbox{ and }p>\max\left(\frac{2q+1}{2},2q\right).$

Observe that in this paper we are able to obtain extra and improved pointwise and integral estimates, which shows again that, though the damping parameters and the Tikhonov regularization parameter are strongly correlated, the best choice of the damping is not $\d\sqrt{\e(t)},\,\d>0.$

Finally, note that explicit Euler discretization of the dynamical system \eqref{DynSys} yields to inertial gradient methods. Indeed,  by discretizing \eqref{DynSys} with the constant stepsize $h$,  $t_n=nh,\,\,x_n=x(t_n),\,\e(t_n)=\e_n$  we get
$$\frac{x_{n+1}-2x_n+x_{n-1}}{h^{2}}+\frac{\a}{h^{q+1}n^{q}}(x_n-x_{n-1})+\n g\left(x_n+\left(\frac{\g}{h}+\frac{\b}{h^{q+1}n^q}\right)(x_n-x_{n-1})\right)+\e_nx_n=0.$$

Equivalently, by denoting $h^{2q}$ with $s$, $h^{1-q}\a$ still with $\a$, $\frac{\g}{h}$ still with $\g,$ $\frac{\b}{h^{q+1}}$ still with $\b$ and  $h^{2q}\e_n$ still with $\e_n$, the latter equation can be written as a gradient type inertial algorithm with a Tikhonov regularization term, that is,
\begin{equation}\label{discrete}
x_{n+1}=x_n+\left(1-\frac{\a}{n^q}\right)(x_n-x_{n-1})-s\n g\left(x_n+\left(\g+\frac{\b}{n^q}\right)(x_n-x_{n-1})\right)-\e_nx_n.
\end{equation}
We emphasize that algorithm \eqref{discrete} has a similar form as the Nesterov type algorithm with double Tikhonov regularization studied in \cite{MKSL}, hence the results from this paper opens the gate for the study of the strong convergence of inertial gradient type algorithms with a Tikhonov regularization term.

\subsection{The organization of the paper}
The outline of the paper is the following. After a short section concerning the existence and uniqueness of the
trajectories generated by the dynamical system \eqref{DynSys}, in section 3 we deal with the setting of the parameters that assure weak convergence of the generated trajectories to a minimizer of the objective function. To this purpose we introduce a general energy functional, which will play the role of a Lyapunov function associated to the dynamical system \eqref{DynSys} and beside the weak convergence of the generated trajectories to a minimizer of the objective function we  obtain fast convergence rates for the function values and velocity but also some integral estimates involving the gradient of the objective function. In section 4 for another constellation of the parameters, by constructing a powerful Lyapunov function we obtain strong convergence of the generated trajectories to the minimal norm minimizer of the objective function, but also fast convergence of the function values in a generated trajectory and fast convergence of the velocity. Some integral estimates are also provided.
In section 5 we conclude our paper and we present some possible related future researches.

\section{Existence and uniqueness of the trajectories generated by the dynamical system \eqref{DynSys}}
In what follows we show the existence and uniqueness of a classical $C^2$ solution $x$ of the dynamical system \eqref{DynSys}. To this purpose we rewrite (\ref{DynSys})  as a first order system relevant for the Cauchy-Lipschitz-Picard theorem.
\begin{theorem}\label{exuniq}
Let $(u_0, v_0) \in \mathcal{H} \times \mathcal{H}$. Then, the dynamical system (\ref{DynSys}) admits a unique global $C^2(]t_0,+\infty[,\mathcal{H})$ solution.
\end{theorem}
\begin{proof}
Indeed, by using the notation  $X(t) := (x(t), \dot{x}(t))$, the dynamical system (\ref{DynSys}) can be put in the  form
\begin{align}\label{DynSysFirstOrder2}
\begin{cases}
\dot{X}(t) = F(t, X(t)) \\
X(t_0) = (u_0, v_0),
\end{cases}
\end{align}
where $F : [t_0, +\infty[ \times \mathcal{H} \times \mathcal{H} \To \mathcal{H} \times \mathcal{H}$,
$
F(t, u, v) = \left( v, -\frac{\alpha}{t^q} v - \e(t)u  - \nabla g\left(u +\b(t)v\right) \right),
$
and we denote $\b(t)=\g+\frac{\b}{t^q}.$ Note that due to our assumptions $\b(t)>0$ on the interval $[t_0,+\infty[.$

 Our proof is inspired from \cite{AC}. Since $\n g$ is Lipschitz on bounded sets, it is obvious that for \eqref{DynSysFirstOrder2} the classical Cauchy-Picard theorem can be applied, hence, there exist a unique  $C^1$ local solution $X.$ Consequently, \eqref{DynSys} has a unique  $C^2$ local solution. Let $x$ be a maximal solution of  \eqref{DynSys}, defined on an interval $[t_0,T_{\max}[,\,T_{\max}\le+\infty.$

 In order to prove that  $\dot{x}$ is bounded  we consider the energy functional $W:[t_0,+\infty[\to\R,$
\begin{align}\label{Efct1}
W(t) = g(x(t)) + \frac{1}{2} \| \dot{x}(t) \|^2 + \frac{\e(t)}{2} \| x(t) \|^2.
\end{align}
The time derivative of \eqref{Efct1} reads as
\begin{align*}
\dot{W}(t) = \langle \nabla g(x(t)), \dot{x}(t) \rangle + \langle \dot{x}(t), \ddot{x}(t) \rangle + \frac{\dot{\e}(t)}{2} \| x(t) \|^2 + \e(t) \langle x(t), \dot{x}(t) \rangle.
\end{align*}
From \eqref{DynSys} we have $\ddot{x}(t) = - \frac{\alpha}{t^q} \dot{x}(t) - \e(t) x(t) - \nabla g(x(t) + \beta(t) \dot{x}(t))$, and we obtain
\begin{align*}
\dot{W}(t) &=  \langle \n g(x(t))-\nabla g(x(t) + \beta(t) \dot{x}(t)), \dot{x}(t) \rangle + \frac{\dot{\e}(t)}{2} \| x(t) \|^2- \frac{\alpha}{t^q} \| \dot{x}(t) \|^2.
\end{align*}
Now,  by using the fact that $\n g$ is monotone and $\dot{\e}(t)\le 0$ we get
\begin{align*}
\dot{W}(t) &=  -\frac{1}{\b(t)}\langle \n g(x(t))-\nabla g(x(t) + \beta(t) \dot{x}(t)),x(t)-( x(t) + \beta(t)\dot{x}(t)) \rangle + \frac{\dot{\e}(t)}{2} \| x(t) \|^2- \frac{\alpha}{t^q} \| \dot{x}(t) \|^2\\
&\le 0.
\end{align*}
Consequently, $W$ is non-increasing on $[t_0, +\infty[$, hence by the fact that $g$ is bounded from below we get $ \| \dot{x}(t) \| < + \infty$ for all $t\ge t_0.$

Let $\|\dot{x}_{\infty}\|=\sup_{t\in [t_0, T_{\max}[}\|\dot{x}(t)\|$
and assume that $T_{\max}<+\infty.$ Since $\|x(t)-x(t')\|\le\|\dot{x}_{\infty}\||t-t'|$, we get that $\lim_{t\to T_{\max}}x(t):= x_{\infty}\in \mathcal{H}$. By \eqref{DynSys} the map $\ddot{x}$ is also bounded on the interval $[t_0, T_{\max}[$ and under the same argument as before
$\lim_{t\to T_{\max}}\dot{x}(t):= \dot{x}_{\infty}$ exists. Applying the local existence theorem with initial data $(x_{\infty},\dot{x}_{\infty})$, we can extend the maximal solution to a strictly larger interval, a clear contradiction. Hence $T_{\max}=+\infty,$ which completes the proof.
\end{proof}

\section{Weak convergence for the trajectories of the regularized dynamical system \eqref{DynSys} and fast convergence for the function values in a generated trajectory }

In this section we show that the function values in a trajectory generated by the dynamical system \eqref{DynSys} converge to the global minimum of the objective function. Further, we obtain  some pointwise and integral estimates concerning the function values in a  trajectory generated by the dynamical system \eqref{DynSys}, but also for the gradient  and velocity.
Finally, we give sufficient conditions on the Tikhonov regularization parameter in order to show that the trajectory $x(t)$ generated by the dynamical system \eqref{DynSys} converges weakly to a minimizer of our objective function $g.$

In this section we consider the following weak assumption:
\begin{align}
  &\text{There exist }0< K \text{ such that }  \frac{q(1-q)}{t^2}\le \e(t) \le \frac{K}{t^{q}},\mbox{ for }t\mbox{ large enough.} \label{C0}\tag{C0}
\end{align}

Our first general result is the following.

\begin{theorem}\label{GeneralConvergenceResult}
For some starting points $u_0, \, v_0\in\mathcal{H}$ let $x : [t_0, \infty) \to \mathcal{H}$ be the unique global solution of  \eqref{DynSys}.
 If the Tikhonov regularization parameter $\e(t)$  satisfies \eqref{C0} and $\int_{t_0}^{+\infty} \frac{\e(t)}{t^q} dt < + \infty$ then
$\lim\limits_{t\to+\infty} g\left(x(t) + \b(t)\dot{x}(t)\right)=\min g.$
In addition, one has that
$\lim\limits_{t\to+\infty} \|\dot{x}(t)\|=0,$
therefore
$\lim\limits_{t\to+\infty} g\left(x(t)\right)=\min g.$
\end{theorem}

\begin{proof}
Consider $x^\ast \in \argmin g$ and denote $g^\ast = g(x^\ast)=\min g$. Consider $b \in ]0, \a[$ and  $a(t)=t^{2q},\,t\ge t_0.$ For simplicity we denote $\b(t)=\g+\frac{\b}{t^q}$  We introduce the  energy functional $\mathcal{E} : [t_0, \infty) \to \mathbb{R},$
\begin{align}\label{Lyapunov}
\mathcal{E}(t) = &a(t) \left( g(x(t)+\beta(t)\dot{x}(t)) - g^\ast \right) +\frac{t^{2q}\e(t)}{2}\|x(t)\|^2+ \frac{1}{2} \| b(x(t)-x^\ast) + t^q \dot{x}(t) \|^2\\
\nonumber& + \frac{b(\alpha-qt^{q-1}-b)}{2} \| x(t) - x^\ast \|^2.
\end{align}
From the choice of $b$, there exists a $t_1\ge t_0$ such that the coefficient $\a - qt^{q-1} - b > 0$ for all $t\ge t_1$, hence the Lyapunov functional $\mathcal{E}(t)\ge 0$  for all $t\ge t_1.$

By taking the time derivative of $\mathcal{E}$ we obtain
\begin{align}\label{engderiv1}
\dot{\mathcal{E}}(t) &=a'(t) \left( g(x(t)+\beta(t)\dot{x}(t)) - g^\ast \right) + a(t) \<\n g(x(t)+\beta(t)\dot{x}(t)),\b(t)\ddot{x}(t)+(\b'(t)+1)\dot{x}(t)\>\\
\nonumber&+\left(qt^{2q-1}\e(t)+\frac{t^{2q}\dot{\e}(t)}{2}\right)\|x(t)\|^2+t^{2q}\e(t)\<\dot{x}(t),x(t)\>+\<(b+qt^{q-1})\dot{x}(t) +t^q\ddot{x}(t),b(x(t)-x^\ast)+t^q \dot{x}(t)\>\\
\nonumber& + b(\alpha-qt^{q-1}-b) \<\dot{x}(t), x(t) - x^\ast\> + \frac{bq(1-q)t^{q-2}}{2} \| x(t) - x^\ast \|^2.
\end{align}
From the dynamical system (\ref{DynSys}), we have that
$
\ddot{x}(t) = - \e(t) x(t)- \frac{\alpha}{t^q} \dot{x}(t) - \nabla g(x(t) + \beta(t) \dot{x}(t)).
$
Hence,
\begin{align}\label{forenergy1}
&a(t) \<\n g(x(t)+\beta(t)\dot{x}(t)),\b(t)\ddot{x}(t)+(\b'(t)+1)\dot{x}(t)\>=\\
\nonumber& a(t)\left\<\n g(x(t)+\b(t)\dot{x}(t)), -\b(t)\e(t)x(t)+\left(-\b(t)\frac{\a}{t^q}+\b'(t)+1\right)\dot{x}(t))-\b(t)\nabla g(x(t) + \beta(t) \dot{x}(t))\right\>=\\
\nonumber& -\b(t) a(t)\left\|\n g(x(t)+\b(t)\dot{x}(t))\right\|^2+\left(-\b(t)\frac{\a}{t^q}+\b'(t)+1\right)a(t)\<\n g(x(t)+\b(t)\dot{x}(t)),\dot{x}(t)\>\\
\nonumber&-\b(t)\e(t)a(t)\<\n g(x(t)+\b(t)\dot{x}(t)),x(t)\>.
\end{align}

Further,
\begin{align}\label{forenergy2}
&\<(b+qt^{q-1})\dot{x}(t) +t^q\ddot{x}(t),b(x(t)-x^\ast)+t^q \dot{x}(t)\>=\\
\nonumber&\<(b+qt^{q-1}-\a)\dot{x}(t) -t^q\e(t)x(t)-t^q\nabla g(x(t) + \beta(t) \dot{x}(t)), b(x(t)-x^\ast)+ t^q\dot{x}(t)\>=\\
\nonumber& b(b+qt^{q-1}-\a)\<\dot{x}(t),x(t)-x^\ast\>+(b+qt^{q-1}-\a)t^q\|\dot{x}(t)\|^2-bt^q\e(t)\<x(t),x(t)-x^\ast\>-t^{2q}\e(t)\<\dot{x}(t),x(t)\>\\
\nonumber& -b t^q\<\n g(x(t) + \b(t) \dot{x}(t)), x(t)-x^\ast\>-t^{2q}\<\nabla g(x(t) + \beta(t) \dot{x}(t)), \dot{x}(t)\>.
\end{align}

Combining \eqref{engderiv1}, \eqref{forenergy1} and \eqref{forenergy2} we get

\begin{align}\label{engderiv2}
\dot{\mathcal{E}}(t) &= a'(t) \left( g(x(t)+\beta(t)\dot{x}(t)) - g^\ast \right) -\b(t) a(t)\left\|\n g(x(t)+\b(t)\dot{x}(t))\right\|^2+(b+qt^{q-1}-\a)t^q\|\dot{x}(t)\|^2\\
 \nonumber& +\left(qt^{2q-1}\e(t)+\frac{t^{2q}\dot{\e}(t)}{2}\right)\|x(t)\|^2+\left(\left(-\b(t)\frac{\a}{t^q}+\b'(t)+1\right)a(t)-t^{2q}\right)\<\n g(x(t)+\b(t)\dot{x}(t)),\dot{x}(t)\>\\
\nonumber&-\b(t)\e(t)a(t)\<\n g(x(t)+\b(t)\dot{x}(t)),x(t)\>-b t^q\<\n g(x(t) + \b(t) \dot{x}(t))+\e(t)x(t), x(t)-x^\ast\> \\
\nonumber& + \frac{bq(1-q)t^{q-2}}{2} \| x(t) - x^\ast \|^2.
\end{align}

Consider now the strongly convex function
$g_t:\mathcal{H}\To\R,\,g_t(x)=g(x)+\frac{\e(t)}{2}\|x\|^2.$
From the gradient inequality we have
$g_t(y)-g_t(x)\ge\<\n g_t(x),y-x\>+\frac{\e(t)}{2}\|x-y\|^2,\mbox{ for all }x,y\in\mathcal{H}.$
Take now $y=x^\ast$ and $x=x(t)+\b(t)\dot{x}(t).$ We get
\begin{align*}
&g(x^\ast)+\frac{\e(t)}{2}\|x^\ast\|^2-g(x(t)+\b(t)\dot{x}(t))-\frac{\e(t)}{2}\|x(t)+\b(t)\dot{x}(t)\|^2\ge\\
&-\<\n g(x(t)+\b(t)\dot{x}(t))+\e(t)(x(t)+\b(t)\dot{x}(t)),x(t)+\b(t)\dot{x}(t)-x^\ast\>+\frac{\e(t)}{2}\|x(t)+\b(t)\dot{x}(t)-x^\ast\|^2.
\end{align*}
Consequently,
\begin{align*}
&-\<\n g(x(t)+\b(t)\dot{x}(t))+\e(t)x(t),x(t)-x^\ast\>-\b(t)\<\n g(x(t)+\b(t)\dot{x}(t))+\e(t)x(t),\dot{x}(t)\>=\\
&-\<\n g(x(t)+\b(t)\dot{x}(t))+\e(t)x(t),x(t)+\b(t)\dot{x}(t)-x^\ast\>\le g(x^\ast)+\frac{\e(t)}{2}\|x^\ast\|^2-g(x(t)+\b(t)\dot{x}(t))\\
&-\frac{\e(t)}{2}\|x(t)+\b(t)\dot{x}(t)\|^2-\frac{\e(t)}{2}\|x(t)+\b(t)\dot{x}(t)-x^\ast\|^2+\b(t)\e(t)\<\dot{x}(t),x(t)+\b(t)\dot{x}(t)-x^\ast\>.
\end{align*}

From here we get
\begin{align}\label{forenergy3}
&-\<\n g(x(t)+\b(t)\dot{x}(t))+\e(t)x(t),x(t)-x^\ast\>\le\\
\nonumber& -(g(x(t)+\b(t)\dot{x}(t))-g(x^\ast))+\frac{\e(t)}{2}\|x^\ast\|^2+ \b(t)\<\n g(x(t)+\b(t)\dot{x}(t)),\dot{x}(t)\>\\
\nonumber&-\frac{\e(t)}{2}\|x(t)+\b(t)\dot{x}(t)\|^2-\frac{\e(t)}{2}\|x(t)+\b(t)\dot{x}(t)-x^\ast\|^2+\b(t)\e(t)\<\dot{x}(t),2 x(t)+\b(t)\dot{x}(t)-x^\ast\>.
\end{align}

Further, an easy computation shows that
\begin{align*}
&-\frac{\e(t)}{2}\|x(t)+\b(t)\dot{x}(t)\|^2-\frac{\e(t)}{2}\|x(t)+\b(t)\dot{x}(t)-x^\ast\|^2+\b(t)\e(t)\<\dot{x}(t),2 x(t)+\b(t)\dot{x}(t)-x^\ast\>=\\
&-\frac{\e(t)}{2}\|x(t)\|^2-\frac{\e(t)}{2}\|x(t)-x^\ast\|^2.
\end{align*}

Hence, \eqref{forenergy3} becomes
\begin{align}\label{forenergy4}
&-\<\n g(x(t)+\b(t)\dot{x}(t))+\e(t)x(t),x(t)-x^\ast\>\le -(g(x(t)+\b(t)\dot{x}(t))-g(x^\ast))-\frac{\e(t)}{2}\|x(t)\|^2\\
\nonumber&-\frac{\e(t)}{2}\|x(t)-x^\ast\|^2+\frac{\e(t)}{2}\|x^\ast\|^2+ \b(t)\<\n g(x(t)+\b(t)\dot{x}(t)),\dot{x}(t)\>.
\end{align}

By multiplying \eqref{forenergy4} with $b t^q$ and injecting in \eqref{engderiv2} we get

\begin{align}\label{engderiv3}
\dot{\mathcal{E}}(t) \le& (a'(t)-bt^q) \left( g(x(t)+\beta(t)\dot{x}(t)) - g^\ast \right) -\b(t) a(t)\left\|\n g(x(t)+\b(t)\dot{x}(t))\right\|^2+(b+qt^{q-1}-\a)t^q\|\dot{x}(t)\|^2\\
\nonumber&+bt^q\frac{\e(t)}{2}\|x^\ast\|^2 +\left(\frac{t^{2q}\dot{\e}(t)}{2}+(2qt^{q-1}-b)\frac{t^q\e(t)}{2}\right)\|x(t)\|^2\\
\nonumber&-\b(t)\e(t)a(t)\<\n g(x(t)+\b(t)\dot{x}(t)),x(t)\> + \frac{bq(1-q)t^{q-2} - bt^q \e(t)}{2} \| x(t) - x^\ast \|^2
\\
 \nonumber&+\left(\left(-\b(t)\frac{\a}{t^q}+\b'(t)+1\right)a(t)-t^{2q}+b\b(t)t^q\right)\<\n g(x(t)+\b(t)\dot{x}(t)),\dot{x}(t)\>.
\end{align}

Further, using the monotonicity of $\n g$ and the fact that $\n g(x^\ast)=0$ we get for all $t\ge t_1$ that
\begin{align}\label{temp51}
-\b(t)\e(t)a(t)\<\n g(x(t)+\b(t)\dot{x}(t)),x(t)\>&=-\b(t)\e(t)a(t)\<\n g(x(t)+\b(t)\dot{x}(t)),x(t)+\b(t)\dot{x}(t)-x^\ast\>\\
\nonumber&+\b(t)\e(t)a(t)\<\n g(x(t)+\b(t)\dot{x}(t)),\b(t)\dot{x}(t)-x^\ast\>\\
\nonumber&\le \b^2(t)\e(t)a(t)\<\n g(x(t)+\b(t)\dot{x}(t)),\dot{x}(t)\>\\
\nonumber&-\b(t)\e(t)a(t)\<\n g(x(t)+\b(t)\dot{x}(t)),x^\ast\>.
\end{align}

Consequently, by injecting \eqref{temp51} in \eqref{engderiv3}  and taking into account that $a(t)=t^{2q}$, it remains to estimate
$$\left(\left(-\b(t)\frac{\a}{t^q}+\b'(t)+\b^2(t)\e(t)\right)a(t)+b\b(t)t^q\right)\<\n g(x(t)+\b(t)\dot{x}(t)),\dot{x}(t)\>$$
and
$$-\b(t)\e(t)a(t)\<\n g(x(t)+\b(t)\dot{x}(t)),x^\ast\>.$$

Now, since $\b(t)>0$, for all $t\ge t_0$ we get that  for all $t\ge  t_1$ and for all $P,R>0$ one has
\begin{align}\label{tempy}
&\left(\left(-\b(t)\frac{\a}{t^q}+\b'(t)+\b^2(t)\e(t)\right)a(t)+b\b(t)t^q\right)\<\n g(x(t)+\b(t)\dot{x}(t)),\dot{x}(t)\>\le\\
\nonumber& (\b(t)(\a-b)t^q+|\b'(t)|t^{2q}+\b^2(t)\e(t)t^{2q})\left(\frac{\|\n g(x(t)+\b(t)\dot{x}(t))\|^2}{2P}+\frac{P\|\dot{x}(t)\|^2}{2}\right)
\end{align}
and
\begin{align}\label{tempy1}
&-\b(t)\e(t)a(t)\<\n g(x(t)+\b(t)\dot{x}(t)),x^\ast\>\le \b(t)\e(t)t^{2q}\left(\frac{Rt^q\|\n g(x(t)+\b(t)\dot{x}(t))\|^2}{2}+\frac{\|x^*\|^2}{2Rt^q}\right)
\end{align}

Hence, \eqref{engderiv3} leads to

\begin{align}\label{engderiv4}
\dot{\mathcal{E}}(t) \le& (2qt^{2q-1}-bt^q) \left( g(x(t)+\beta(t)\dot{x}(t)) - g^\ast \right)\\
 \nonumber&+\left((b+qt^{q-1}-\a)t^q+\frac{P(\b(t)(\a-b)t^q+|\b'(t)|t^{2q}+\b^2(t)\e(t)t^{2q})}{2}\right)\|\dot{x}(t)\|^2\\
\nonumber&+\left(-\b(t) t^{2q}+\frac{\b(t)(\a-b)t^q+|\b'(t)|t^{2q}+\b^2(t)\e(t)t^{2q}}{2P}+\frac{R\b(t)\e(t)t^{3q}}{2}\right)\left\|\n g(x(t)+\b(t)\dot{x}(t))\right\|^2\\
 \nonumber& +\left(\frac{t^{2q}\dot{\e}(t)}{2}+(2qt^{q-1}-b)\frac{t^q\e(t)}{2}\right)\|x(t)\|^2 + \frac{bq(1-q)t^{q-2} - bt^q \e(t)}{2} \| x(t) - x^\ast \|^2
\\
 \nonumber&+\left(\frac{\b(t)}{2R}+\frac{b}{2}\right)t^q\e(t)\|x^\ast\|^2\mbox{ for all }t\ge t_1,\,P,R>0.
\end{align}

Now, according to \eqref{C0}   there exist $K>0$  and $\ol t_1\ge t_1$ such that $\e(t)\le\frac{K}{t^q} \mbox{ for every }t\ge \ol t_1$ and taking into account that  $\b(t)=\g+\frac{\b}{t^q}$ with $\g>0,\,\b\in\R$ or $\g=0,\,\b>0$ and $q\in]0,1[$ we conclude the following. \\

There exist $t_1'\ge\ol t_1$ and $r_0>0$ such that
$$2qt^{2q-1}-bt^q\le-r_0 t^q\mbox{ for all }t\ge t_1'.$$

One can choose $P_0>0$ such that for some $t_1''\ge t_1'$ and $r_1>0$ one has
$$(b+qt^{q-1}-\a)t^q+\frac{P_0(\b(t)(\a-b)t^q+|\b'(t)|t^{2q}+\b^2(t)\e(t)t^{2q})}{2}\le-r_1t^q\mbox{ for all }t\ge t_1''.$$

If $\g>0,\,\b\in\R$ then one can choose $R_0>0$ such that for some $t_1'''\ge t_1''$ and $r_2>0$ one has
$$-\b(t) t^{2q}+\frac{\b(t)(\a-b)t^q+|\b'(t)|t^{2q}+\b^2(t)\e(t)t^{2q}}{2P_0}+\frac{R_0\b(t)\e(t)t^{3q}}{2}\le -r_2 t^{2q},\mbox{ for all }t\ge t_1'''.$$

If $\g=0,\b>0$ then one can choose $R_0>0$ such that for some $t_1'''\ge t_1''$ and $r_2>0$ one has
$$-\b(t) t^{2q}+\frac{\b(t)(\a-b)t^q+|\b'(t)|t^{2q}+\b^2(t)\e(t)t^{2q}}{2P_0}+\frac{R_0\b(t)\e(t)t^{3q}}{2}\le -r_2 t^{q},\mbox{ for all }t\ge t_1'''.$$

Finally, since $\e(t)$ is decreasing  we get
$$\frac{t^{2q}\dot{\e}(t)}{2} + (2qt^{q-1}-b)\frac{t^q\e(t)}{2} \le (2qt^{q-1}-b)\frac{t^q\e(t)}{2}\le-\frac{r_0}{2}t^q\e(t),\mbox{ for all }t\ge t_1'.$$

Further, using that $\e(t) \ge \frac{q(1-q)}{t^2}$ for $t$ big enough we obtain that there exists $t_2\ge t_1'''$ and $r_3>0$ such that
$\frac{bq(1-q)t^{q-2} - bt^q \e(t)}{2}\le 0$ and $\frac{\b(t)}{2R_0}+\frac{b}{2}<r_3$ for all $t\ge t_2.$

Hence, by denoting $r_2(t)=r_2 t^{2q}$ when $\g>0,\,\b\in\R$ and $r_2(t)=r_2t^{q}$ whenever $\g=0,\,\b>0,$  relation \eqref{engderiv4} leads to
\begin{align}\label{engderiv5}
\dot{\mathcal{E}}(t)& +r_0t^q\left( g(x(t)+\beta(t)\dot{x}(t)) - g^\ast \right)+r_1 t^q\|\dot{x}(t)\|^2+ r_2(t)\left\|\n g(x(t)+\b(t)\dot{x}(t))\right\|^2\\
\nonumber&+\frac{r_0}{2} t^q\e(t)\|x(t)\|^2\le r_3t^q\e(t)\|x^\ast\|^2,\mbox{ for all }t\ge  t_2.
\end{align}

 By integrating \eqref{engderiv5} on the interval $[t_2,t]$, we obtain for every $t\ge t_2$ that
\begin{align}\label{engfinal}
{\mathcal{E}}(t) &+ Cr_0\int_{t_2}^t s^q \left( g(x(s)+\beta(s)\dot{x}(s)) - g^\ast \right)ds+ r_1\int_{t_2}^t s^q\|\dot{x}(s)\|^2ds + \int_{t_2}^t r_2(s)\left\|\n g(x(s)+\b(s)\dot{x}(s))\right\|^2ds\\
\nonumber&+\frac{r_0}{2}\int_{t_2}^t s^q\e(s)\|x(s)\|^2ds\le \frac{r_3\|x^\ast\|^2}{2}\int_{t_2}^{t}  s^q\e(s) ds+\mathcal{E}(t_2).
\end{align}
Now if we assume that $\int_{t_0}^{+\infty}  \frac{\e(s)}{s^q} ds<+\infty,$ then from \eqref{Lyapunov} and from \eqref{engfinal} we get that for all $t\ge t_2$ one has
$$g(x(t)+\beta(t)\dot{x}(t))-\min g \le \frac{\mathcal{E}(t_2)}{t^{2q}} + \frac{r_3\|x^\ast\|^2}{2t^{2q}}\int_{t_2}^t s^q\e(s)ds,$$
$$\left \| \frac{b}{t^q}(x(t)-x^\ast) + \dot{x}(t) \right\|^2\le \frac{2\mathcal{E}(t_2)}{t^{2q}}+\frac{r_3 \|x^\ast\|^2 }{t^{2q}}\int_{t_2}^{t} s^q \e(s)ds$$
and
$$\frac{b(\alpha-qt^{q-1}-b)}{2}\left\|\frac{ x(t) - x^\ast}{t^q}\right\|^2\le \frac{\mathcal{E}(t_2)}{t^{2q}} + \frac{r_3\|x^\ast\|^2}{2t^{2q}}\int_{t_2}^t s^q\e(s)ds.$$

Obviously,
$\lim\limits_{t\to+\infty}\frac{\mathcal{E}(t_3)}{t^{2q}}=0.$
Further, Lemma \eqref{nullimit} applied to the functions $\varphi(s)=s^{2q}$ and $f(s)=\frac{\e(s)}{s^q}$ provides
$\lim\limits_{t\to+\infty}\frac{1}{t^{2q}}\int_{t_2}^t s^{2q}\frac{\e(s)}{s^q}dt=0.$ Hence,
$$\lim\limits_{t\to+\infty}g(x(t)+\beta(t)\dot{x}(t))=\min g,$$
$$\lim\limits_{t\to+\infty}\left \| \frac{b}{t^q}(x(t)-x^\ast) + \dot{x}(t) \right\|=0$$
and
$$\lim\limits_{t\to+\infty}\left\| \frac{x(t) - x^\ast}{t^q}\right\|=0.$$
Combining the last two relations we get
$\lim\limits_{t\to+\infty}\left \| \dot{x}(t) \right\|=0,$
and from here and the continuity of $g$ we have
$$\lim\limits_{t\to+\infty}g(x(t))=\lim\limits_{t\to+\infty}g(x(t)+\beta(t)\dot{x}(t))=\min g.$$
\end{proof}

\begin{remark}\label{remcond} In the hypotheses of Theorem \ref{GeneralConvergenceResult} we assumed that the Tikhonov regularization parameter $\e(t)$  satisfies \eqref{C0} and $\int_{t_0}^{+\infty} \frac{\e(t)}{t^q} dt < + \infty.$ Note that if $q>\frac12$ then \eqref{C0} implies that $\int_{t_0}^{+\infty} \frac{\e(t)}{t^q} dt < + \infty.$ 

Indeed, according to \eqref{C0} there exists $K>0$ and $t_1\ge t_0$ such that $\e(t)\le\frac{K}{t^q}$ for all $t\ge t_1.$ Hence,
$$\int_{t_0}^{+\infty} \frac{\e(t)}{t^q} dt =\int_{t_0}^{t_1} \frac{\e(t)}{t^q} dt +\int_{t_1}^{+\infty} \frac{\e(t)}{t^q} dt\le\int_{t_0}^{t_1} \frac{\e(t)}{t^q} dt +\int_{t_1}^{+\infty} \frac{K}{t^{2q}} dt < + \infty.$$

When we specify $\e(t)=\frac{a}{t^p},\,a,p>0$ then \eqref{C0} is nothing else that $q\le p\le 2$ and if $p=2$ then $a\ge q(1-q).$ In this case $\int_{t_0}^{+\infty} \frac{\e(t)}{t^q} dt < + \infty$ whenever $p>1-q.$ 

Hence, for $\max(q,1-q)<p\le 2$ and if $p=2$ then $a\ge q(1-q)$ then the conclusion of Theorem \ref{GeneralConvergenceResult} remains valid.
\end{remark}

In what follows, under some stronger conditions imposed on the Tikhonov regularization parameter, we obtain pointwise and integral estimates for the function values, velocity and gradient.

\begin{theorem}\label{RateConvergenceResult}
For some starting points $u_0, \, v_0\in\mathcal{H}$ let $x : [t_0, \infty) \to \mathcal{H}$ be the unique global solution of  \eqref{DynSys}. Suppose that  the condition \eqref{C0} holds and the Tikhonov regularization parameter satisfies $\int_{t_0}^{+\infty} t^q \e(t) dt < + \infty$. 
Then, the following integral and pointwise estimates are valid.

One has that $\int_{t_0}^{+\infty} t^q \|\dot{x}(t)\|^2dt<+\infty \text{ and } \int_{t_0}^{+\infty} t^q \left( g\left(x(t) + \b(t)\dot{x}(t)\right) - \min g\right)dt<+\infty.$ Further,

$\int_{t_0}^{+\infty} t^{2q}\left\|\n g\left(x(t) + \b(t)\dot{x}(t)\right)\right\|^2dt<+\infty,$ if $\g>0,\,\b\in\R$
and $ \int_{t_0}^{+\infty} t^{q} \left\|\n g\left(x(t) + \b(t)\dot{x}(t)\right)\right\|^2dt<+\infty,$ whenever  $\g=0,\,\b>0.$

Moreover,
 $\| \dot{x}(t) \|=o\left(\frac{1}{t^q}\right)\mbox{ as } t\to+\infty$,
$ g\left(x(t) + \b(t)\dot{x}(t)\right) - \min g =o\left(\frac{1}{t^{2q}}\right) \text{ as } t\to+\infty$
and
$\|x(t)\|=o\left(\frac{1}{t^q\sqrt{\e(t)}}\right) \mbox{ as } t\to+\infty.$
\end{theorem}

\begin{proof} First we evaluate the right hand side of \eqref{engfinal}. By using the assumption $\int_{t_0}^{+\infty}  s^q\e(s) ds<+\infty$ we conclude that $\frac{r_3}{2} \|x^\ast\|^2 \int_{t_2}^{+\infty}  t^q \e(t) dt+\mathcal{E}(t_2)<+\infty$, a hence \eqref{engfinal} leads to 
\begin{align}\label{engfinal1}
{\mathcal{E}}(t) &+ Cr_0\int_{t_2}^t s^q \left( g(x(s)+\beta(s)\dot{x}(s)) - g^\ast \right)ds+ r_1\int_{t_2}^t s^q\|\dot{x}(s)\|^2ds + \int_{t_2}^t r_2(s)\left\|\n g(x(s)+\b(s)\dot{x}(s))\right\|^2ds\\
\nonumber&+\frac{r_0}{2}\int_{t_2}^t s^q\e(s)\|x(s)\|^2ds\le C_0,
\end{align}
for some $C_0>0.$

Now, from \eqref{engfinal1} we immediately deduce that
${\mathcal{E}}(t)<+\infty$ for all $t\ge t_0.$ Hence, taking into account the definition of $\mathcal{E}$  we  obtain
$ g(x(t)+\beta(t)\dot{x}(t)) - g^* =\mathcal{O}\left(\frac{1}{t^{2q}}\right),\mbox{ as }t\to+\infty,$
$\sup_{t\ge t_0}\| b(x(t)-x^\ast) + t^q \dot{x}(t) \|^2<+\infty,$ and $ \sup_{t\ge t_0} \| x(t) - x^\ast \|^2<+\infty.$

The latter two relations give us at once $\|\dot{x}(t)\|=\mathcal{O}\left(\frac{1}{t^{2q}}\right),\mbox{ as }t\to+\infty$ and that the trajectory $x(t)$ is bounded.

Further, \eqref{engfinal1} yields the following integral estimates.
\begin{align}
& \int_{t_0}^{+\infty} t^q \left( g(x(t)+\beta(t)\dot{x}(t)) - g^\ast \right)dt<+\infty,\label{fordecay}\\
& \int_{t_0}^{+\infty} t^q\|\dot{x}(t)\|^2dt<+\infty, \label{forvelocity}\\
& \int_{t_0}^{+\infty} t^{2q}\left\|\n g(x(t)+\b(t)\dot{x}(t))\right\|^2dt<+\infty,\mbox{ whenever }\
\g>0,\,\b\in\R,\label{forgradient}\\
& \int_{t_0}^{+\infty} t^{q}\left\|\n g(x(t)+\b(t)\dot{x}(t))\right\|^2dt<+\infty,\mbox{ whenever }\
\g=0,\,\b> 0,\label{forgradient1}\\
& \int_{t_0}^{+\infty} t^q\e(t)\|x(t)\|^2dt<+\infty\label{fortraj}
\end{align}

In order to show that $ g(x(t)+\beta(t)\dot{x}(t)) - g^* =o\left(\frac{1}{t^{2q}}\right),\mbox{ as }t\to+\infty,$ and $\|\dot{x}(t)\|=o\left(\dfrac{1}{t^q}\right),\mbox{ as }t\to+\infty$ assume for now that the limit $\lim\limits_{t\to+\infty} \|x(t)-x^\ast\|$ exists, which shall be proved in Theorem \ref{wlimit}. Then, \eqref{engderiv5} provides that
$$\dot{\mathcal{E}}(t)\le \frac{r_3\|x^\ast\|^2}{2} t^q\e(t),\mbox{ for all }t\ge t_2.$$
 Obviously, by the hypotheses we have $\dfrac{r_3\|x^\ast\|^2}{2} t^q\e(t)\in L^1([t_2,+\infty[)$, hence, according to Lemma \eqref{fejer-cont1} there exists the limit
$\lim\limits_{t\to+\infty}\mathcal{E}(t).$

 Since the limit $\lim\limits_{t\to+\infty} \|x(t)-x^\ast\|$ exists we get that the limit
\begin{equation}\label{temp7}
\lim\limits_{t\to+\infty}t^{2q} \left( g(x(t)+\beta(t)\dot{x}(t)) - g^\ast \right) +\frac{t^{2q}\e(t)}{2}\|x(t)\|^2+ \frac{1}{2} \|t^q \dot{x}(t) \|^2
\end{equation}
also exists. \\
By combining \eqref{fordecay}, \eqref{forvelocity} and \eqref{fortraj} we find that
\begin{equation}\label{temp8}
\int_{t_0}^{+\infty}\frac{1}{t^q}\left(t^{2q} \left( g(x(t)+\beta(t)\dot{x}(t)) - g^\ast \right) +\frac{t^{2q}\e(t)}{2}\|x(t)\|^2+ \frac{1}{2} \|t^q \dot{x}(t) \|^2\right)dt<+\infty.
\end{equation}
Since $q < 1$, the function $t\To \frac{1}{t^q}\not\in L^1([t_0,+\infty[)$, therefore \eqref{temp7} and \eqref{temp8} lead to
\begin{equation}\label{temp9}
\lim\limits_{t\to+\infty}t^{2q} \left( g(x(t)+\beta(t)\dot{x}(t)) - g^\ast \right) +\frac{t^{2q}\e(t)}{2}\|x(t)\|^2+ \frac{1}{2} \|t^q \dot{x}(t) \|^2=0.
\end{equation}
Consequently,
$g(x(t)+\beta(t)\dot{x}(t)) -\min g=o\left(\frac{1}{t^{2q}}\right)\mbox{ as }t\to+\infty,$
$\|x(t)\|=o\left(\frac{1}{t^q\sqrt{\e(t)}}\right)\mbox{ as }t\to+\infty,$
and
$\|\dot{x}(t)\|=o\left(\frac{1}{t^q}\right)\mbox{ as }t\to+\infty.$
\end{proof}

In what follows we show that under the hypotheses of Theorem \ref{RateConvergenceResult}, the trajectories generated by the dynamical system \eqref{DynSys} converge weakly to a minimizer of the objective function $g.$

\begin{theorem}\label{wlimit}
For some starting points $u_0, \, v_0\in\mathcal{H}$ let $x : [t_0, \infty) \to \mathcal{H}$ be the unique global solution of  \eqref{DynSys}. Suppose that  the condition \eqref{C0} holds and the Tikhonov regularization parameter satisfies $\int_{t_0}^{+\infty} t^q \e(t) dt < + \infty$.

Then  the trajectory $x(t)$ converges weakly, as $t\to+\infty$, to an element of $\argmin g$. 
\end{theorem}
\begin{proof} We will use Opial's lemma (see Lemma \ref{Opial} from Appendix).
 In order to prove the existence of the weak limit of $x(t)$, we will use Lemma 16 from \cite{L-jde} (see Lemma \ref{forwconv} from Appendix) and Lemma 6 from \cite{L-jde}. For $x^\ast \in\argmin g$ let us introduce the anchor function
$w(t)=\frac12\|x(t)-x^\ast\|^2.$
The classical derivation chain rule leads to
$$\ddot{w}(t)+\frac{\a}{t^q}\dot{w}(t)=\left\<\ddot{x}(t)+\frac{\a}{t^q}\dot{x}(t),x(t)-x^\ast\right\>+\|\dot{x}(t)\|^2=\left\<-\n g(x(t)+\beta(t)\dot{x}(t))-\e(t){x}(t),x(t)-x^\ast\right\>+\|\dot{x}(t)\|^2.$$
We consider
$$ \< -\n g(x(t)+\beta(t)\dot{x}(t)), x(t) - x^\ast \> = \< -\n g(x(t)+\beta(t)\dot{x}(t)), x(t) + \b(t) \dot{x}(t) - x^\ast \> + \b(t) \< \n g(x(t)+\beta(t)\dot{x}(t)), \dot{x}(t) \>.$$
From the monotonicity of $\n g$ we have that
$$\< -\n g(x(t)+\beta(t)\dot{x}(t)), x(t) + \b(t) \dot{x}(t) - x^\ast \>\le 0,$$
from which it follows that
$$ \< -\n g(x(t)+\beta(t)\dot{x}(t)), x(t) - x^\ast \> \le \b(t) \< \n g(x(t)+\beta(t)\dot{x}(t)), \dot{x}(t) \> \le \frac{|\b(t)|}{2} \left( \| \n g(x(t)+\beta(t)\dot{x}(t)) \|^2 + \| \dot{x}(t) \|^2 \right). $$
Using that $|\b(t)| \le \g + \dfrac{|\b|}{t} \le \Tilde{K} = \g + \dfrac{|\b|}{t_0}$, it follows that
$$ \< -\n g(x(t)+\beta(t)\dot{x}(t)), x(t) - x^\ast \> \le \frac{\Tilde{K}}{2} \left( \| \n g(x(t)+\beta(t)\dot{x}(t)) \|^2 + \| \dot{x}(t) \|^2 \right).$$
By employing that
$$\left\<-\e(t)x(t),x(t)-x^\ast\right\>=\frac{\e(t)}{2}(\|x^\ast\|^2-\|x(t)\|^2-\|x(t)-x^\ast\|^2),$$
and using that $\e(t) \in \mathbb{R}_{+}$, it follows that
\begin{equation}\label{foruse}
\ddot{w}(t)+\frac{\a}{t^q}\dot{w}(t)\le \frac{\e(t)}{2}\|x^\ast\|^2 + \left( 1 + \frac{\Tilde{K}}{2} \right) \|\dot{x}(t)\|^2 + \dfrac{|\b(t)|}{2} \| \n g(x(t)+\beta(t)\dot{x}(t)) \|^2.
\end{equation}
In what follows, let us consider $k(t)=\frac{\e(t)}{2}\|x^\ast\|^2+\left( 1 + \frac{\Tilde{K}}{2} \right) \|\dot{x}(t)\|^2 + \dfrac{\Tilde{K}}{2} \| \n g(x(t)+\beta(t)\dot{x}(t)) \|^2$. Note that  by the hypotheses of the theorem  we have $t^q \frac{\e(t)}{2}\|x^\ast\|^2\in L^1[t_0,+\infty[$.
 Further, using the integral estimate \eqref{forvelocity} it follows that $t^q \| \dot{x}(t) \|^2 \in L^1[t_0,+\infty[$, hence $\left( 1 + \frac{\Tilde{K}}{2} \right) t^q \| \dot{x}(t) \|^2 \in L^1[t_0,+\infty[$.
Moreover, in one hand for the case when $\g=0, \b > 0$, the integral estimate \eqref{forgradient1} leads to $\b t^{q} \| \n g(x(t)+\beta(t)\dot{x}(t)) \|^2 \in L^1[t_0,+\infty[$. By taking into account that in this case $t^q|\b(t)|=\b $ we find that $\frac{|\b(t)|}{2} t^{q} \| \n g(x(t)+\beta(t)\dot{x}(t)) \|^2 \in L^1[t_0,+\infty[$.
On the other hand, when $\g >0, \b \in \mathbb{R}$, from \eqref{forgradient} we have that $t^{2q} \| \n g(x(t)+\beta(t)\dot{x}(t)) \|^2 \in L^1[t_0,+\infty[$, which evidently implies that $t^q \| \n g(x(t)+\beta(t)\dot{x}(t)) \|^2 \in L^1[t_0,+\infty[$, and  we finally obtain $\dfrac{|\b(t)|}{2} t^q \| \n g(x(t)+\beta(t)\dot{x}(t)) \|^2 \in L^1[t_0,+\infty[$.

 Hence, we get that $t^q k(t) \in L^1[t_0, +\infty[$. Now, by applying the same reasoning as in Lemma 6 from \cite{L-jde}, we obtain the existence of the limit $\lim\limits_{t \to +\infty} \| x(t) - x^\ast \|$. Now, if $\ol x \in  \mathcal{H}$ is  a weak sequential limit point of $x(t)$ then there exists a
sequence $(t_n)_{n\in \N} \subseteq  [t_0,+\infty[$ such that $\lim_{n\to+ \infty}  t_n = +\infty$  and $x(t_n)$ converges weakly
to $\ol x$ as $n\to+\infty.$
Obviously the function $g$ is weakly lower semicontinuous, since it is convex
and continuous, consequently $ g(\ol x) \leq  \liminf_{n\to+\infty}  g(x(t_n))=\lim_{n\to+\infty}  g(x(t_n))=g^*= \min g$,  which shows that $\ol x \in  \argmin g.$
According to the Opial lemma it follows that
$$w-\lim_{t\to+\infty}x(t) \in  \argmin g.$$
\end{proof}

\begin{remark}\label{rcond} Let us  consider $\e(t) = \frac{a}{t^p}$ where $a, \, p>0$. Then, as we have seen before, condition (C0) is satisfied whenever $2\ge p\ge q$, and when $p = 2$, then $a \ge q(1-q)$. Now, the condition $\int_{t_0}^{+\infty} t^q \e(t) dt < +\infty$ leads to $p > 1 + q$, thus for $p \in ]1+q, \, 2]$ the conclusions of Theorem \ref{RateConvergenceResult} and Theorem \ref{wlimit} hold. Note that this fact is in concordance with the results obtained in \cite{L-jde} and \cite{BCL-jma}. 
\end{remark}

Consequently the following result is valid.
\begin{corollary}\label{specweak} Consider $0<q<1$ and $\e(t)=\frac{a}{t^p},\,a>0$ and assume that $q+1< p\le 2.$ Assume further that for $p=2$ then $a \ge q(1-q)$. For some starting points $u_0, \, v_0\in\mathcal{H}$ let $x : [t_0, +\infty[ \to \mathcal{H}$ be the unique global solution of  \eqref{DynSys}. Then, the trajectory $x(t)$ converges weakly, as $t\to+\infty$, to an element of $\argmin g$. Further, one has $\| \dot{x}(t) \|=o\left(t^{-q}\right),$
$ g\left(x(t) + \b(t)\dot{x}(t)\right) - \min g =o\left(t^{-2q}\right) \text{ as } t\to+\infty$ and $\int_{t_0}^{+\infty} t^q \|\dot{x}(t)\|^2dt<+\infty,\, \int_{t_0}^{+\infty} t^q \left( g\left(x(t) + \b(t)\dot{x}(t)\right) - \min g\right)dt<+\infty.$ Concerning the gradient one has the integral estimates\\
$\int_{t_0}^{+\infty} t^{2q}\left\|\n g\left(x(t) + \b(t)\dot{x}(t)\right)\right\|^2dt<+\infty,$ if $\g>0,\,\b\in\R$
and
$ \int_{t_0}^{+\infty} t^{2q-1} \left\|\n g\left(x(t) + \b(t)\dot{x}(t)\right)\right\|^2dt<+\infty,$ if $\g=0,\,\b>0.$
\end{corollary}
Next we show that even in case $p=q+1$ the generated trajectory $x$ is bounded and some pointwise estimates hold.

\begin{theorem}\label{pq1} Consider $0<q<1$ and let $\e(t)=\frac{a}{t^{q+1}},\,a>0.$ For some starting points $u_0, \, v_0\in\mathcal{H}$ let $x : [t_0, +\infty[ \to \mathcal{H}$ be the unique global solution of  \eqref{DynSys}. Then, the trajectory $x(t)$ is bounded and one has $\| \dot{x}(t) \|=\mathcal{O}\left(t^{-q}\right),$
$ g\left(x(t) + \b(t)\dot{x}(t)\right) - \min g =\mathcal{O}\left(t^{-2q}\right) \text{ as } t\to+\infty.$ 

\end{theorem}
\begin{proof} We consider the energy functional defined in the proof of Theorem \ref{GeneralConvergenceResult}. In this special case \eqref{engderiv4} reads as
\begin{align}\label{engderiv4pq1}
\dot{\mathcal{E}}(t) \le& (2qt^{2q-1}-bt^q) \left( g(x(t)+\beta(t)\dot{x}(t)) - g^\ast \right)\\
 \nonumber&+\left((b+qt^{q-1}-\a)t^q+\frac{P(\b(t)(\a-b)t^q+|\b'(t)|t^{2q}+a\b^2(t)t^{q-1})}{2}\right)\|\dot{x}(t)\|^2\\
\nonumber&+\left(-\b(t) t^{2q}+\frac{\b(t)(\a-b)t^q+|\b'(t)|t^{2q}+a\b^2(t)t^{q-1}}{2P}+\frac{aR\b(t)t^{2q-1}}{2}\right)\left\|\n g(x(t)+\b(t)\dot{x}(t))\right\|^2\\
 \nonumber& +\left(\frac{-a(q+1)t^{q-2}}{2}+(2qt^{q-1}-b)\frac{at^{-1}}{2}\right)\|x(t)\|^2 + \frac{bq(1-q)t^{q-2} -a bt^{-1}}{2} \| x(t) - x^\ast \|^2
\\
 \nonumber&+\left(\frac{\b(t)}{2R}+\frac{b}{2}\right)at^{-1}\|x^\ast\|^2\mbox{ for all }t\ge t_1,\,P,R>0.
\end{align}

Further,

\begin{align}\label{Lyapunovpq1}
\mathcal{E}(t) &= t^{2q} \left( g(x(t)+\beta(t)\dot{x}(t)) - g^\ast \right) +\frac{at^{q-1}}{2}\|x(t)\|^2+ \frac{1}{2} \| b(x(t)-x^\ast) + t^q \dot{x}(t) \|^2 + \frac{b(\alpha-qt^{q-1}-b)}{2} \| x(t) - x^\ast \|^2\\
\nonumber& \le t^{2q} \left( g(x(t)+\beta(t)\dot{x}(t)) - g^\ast \right) +\frac{at^{q-1}}{2}\|x(t)\|^2+ t^{2q} \| \dot{x}(t) \|^2 + \frac{b(\alpha-qt^{q-1}+b)}{2} \| x(t) - x^\ast \|^2,
\end{align}
for all $t\ge t_1.$
Let us consider now $L>0$ which will be specified later. Then, \eqref{engderiv4pq1} and \eqref{Lyapunovpq1} lead to

\begin{align}\label{engderiv5pq1}
\dot{\mathcal{E}}(t)+\frac{L}{t}\mathcal{E}(t) \le& ((2q+L)t^{2q-1}-bt^q) \left( g(x(t)+\beta(t)\dot{x}(t)) - g^\ast \right)\\
 \nonumber&+\left((b+(q+L)t^{q-1}-\a)t^q+\frac{P(\b(t)(\a-b)t^q+|\b'(t)|t^{2q}+a\b^2(t)t^{q-1})}{2}\right)\|\dot{x}(t)\|^2\\
\nonumber&+\left(-\b(t) t^{2q}+\frac{\b(t)(\a-b)t^q+|\b'(t)|t^{2q}+a\b^2(t)t^{q-1}}{2P}+\frac{aR\b(t)t^{2q-1}}{2}\right)\left\|\n g(x(t)+\b(t)\dot{x}(t))\right\|^2\\
 \nonumber& +\left(\frac{a(L+q-1)t^{q-2}}{2}-\frac{abt^{-1}}{2}\right)\|x(t)\|^2 \\
 \nonumber&+ \frac{bq(1-q-L)t^{q-2} +b((\a+b)L-a)t^{-1}}{2}\| x(t) - x^\ast \|^2
\\
 \nonumber&+\left(\frac{\b(t)}{2R}+\frac{b}{2}\right)at^{-1}\|x^\ast\|^2\mbox{ for all }t\ge t_1,\,P,R>0.
\end{align}

Now,  by taking into account that  $\b(t)=\g+\frac{\b}{t^q}$ with $\g>0,\,\b\in\R$ or $\g=0,\,\b>0$ and $q\in]0,1[$ we conclude the following. \\

There exist $t_1'\ge t_1$ and $r_0>0$ such that
$$(2q+L)t^{2q-1}-bt^q\le0\mbox{ for all }t\ge t_1'.$$

One can choose $P_0>0$ such that for some $t_1''\ge t_1'$  one has
$$(b+(q+L)t^{q-1}-\a)t^q+\frac{P_0(\b(t)(\a-b)t^q+|\b'(t)|t^{2q}+a\b^2(t)t^{q-1})}{2}\le0\mbox{ for all }t\ge t_1''.$$

One can choose $R_0>0$ such that for some $t_1'''\ge t_1''$  one has
$$-\b(t) t^{2q}+\frac{\b(t)(\a-b)t^q+|\b'(t)|t^{2q}+a\b^2(t)t^{q-1}}{2P_0}+\frac{aR_0\b(t)t^{2q-1}}{2}\le 0,\mbox{ for all }t\ge t_1'''.$$

Finally,  let $L<\frac{a}{\a+b}.$ Then there exists $t_2\ge t_1'''$ such that
$$\frac{a(L+q-1)t^{q-2}}{2}-\frac{abt^{-1}}{2}\le0,\mbox{ for all }t\ge t_2$$
and
$$bq(1-q-L)t^{q-2} +b((\a+b)L-a)t^{-1}\le0,\mbox{ for all }t\ge t_2.$$

Consequently, by denoting $\sup_{t\ge t_2}\left(\frac{\b(t)}{2R_0}+\frac{b}{2}\right)a\|x^\ast\|^2=M$, \eqref{engderiv5pq1} leads to

\begin{align}\label{engderiv6pq1}
\dot{\mathcal{E}}(t)+\frac{L}{t}\mathcal{E}(t) \le \frac{M}{t}\mbox{ for all }t\ge t_2.
\end{align}
We multiply \eqref{engderiv6pq1} with $t^L$ then integrate on an interval $[t_2,T]$ in order to obtain
$$\mathcal{E}(T)\le \frac{M}{L}+\frac{t_2^L\mathcal{E}(t_2)-\frac{M}{L}t_2^L}{T^L},$$
that is $\mathcal{E}(t)$ is bounded. Taking into account the definition of $\mathcal{E}(t)$ we obtain at once that
$x(t)$ is bounded, $g(x(t)+\beta(t)\dot{x}(t)) - g^\ast=\mathcal{O}(t^{-2q})$ and $\|\dot{x}(t)\|=\mathcal{O}(t^{-q})$ as $t\to+\infty.$
\end{proof}

\section{On the strong convergence of the trajectories generated by the dynamical system \eqref{DynSys}}

In the present section we will recall the idea of Tikhonov regularization, which is linked to the strong convergence results to a minimizer of the objective function of minimal norm regarding our dynamical system \eqref{DynSys}. Let us denote $x_{t}$ the unique solution of the strongly convex minimization problem $ \min_{x \in \mathcal{H}} g_t(x),$
where $g_t(x) =  g(x) + \frac{\e(t)}{2} \| x \|^2 $. From \cite{att-com1996} we know that the Tikhonov approximation curve $t \to x_{t}$ satisfies $x^\ast = \lim\limits_{t \to +\infty} x_{t}$, where $x^\ast = \argmin\limits_{x \in \argmin g} \| x \|$ is the minimal norm element from the set $\argmin g.$ Therefore, $x^\ast=\proj_{\argmin g} 0$ and $\| x_t \| \leq \| x^\ast \|$ (see \cite{BCL}). By employing the definition of the strongly convex function $g_t(x)$ we get that
\begin{equation}\label{fontos0general}
\n g_t(y)=\n g(y)+\e(t)y,
\end{equation}
and by taking into account the unique minimum denoted as $x_{t}$, it follows that
\begin{equation}\label{fontos0}
\n g_t(x_{t})=\n g(x_{t})+\e(t)x_{t}=0.
\end{equation}
As claimed in Lemma 2 from \cite{ABCR} the derivable almost everywhere function $t\To x_{t}$ is endowed with the property
\begin{equation}\label{fontos1}
\left\|\frac{d}{dt}x_{t}\right\| \le -\frac{\dot{\e}(t)}{\e(t)} \|x_{t}\| \mbox{ for almost every } t\ge t_0.
\end{equation}
Because $g_t$ is strongly convex, the gradient inequality leads to
\begin{equation}\label{fontos2}
g_t(y)-g_t(x)\ge\<\n g_t(x),y-x\>+\frac{\e(t)}{2}\|x-y\|^2,\mbox{ for all }x,y\in\mathcal{H}.
\end{equation}
In particular
\begin{equation}\label{fontos3}
g_t(x)-g_t(x_t)\ge\frac{\e(t)}{2}\|x-x_t\|^2,\mbox{ for all }x\in\mathcal{H}.
\end{equation}
Let $y:[t_0,+\infty[\to\mathcal{H}, s\To y(s)$ derivable at a point $t\in]t_0,+\infty[.$ Then, one has that
\begin{equation}\label{fontos4}
\frac{d}{dt}g_t(y(t))=\<\n g_t(y(t)),\dot{y}(t)\>
+\frac{\dot{\e}(t)}{2}\|y(t)\|^2.
\end{equation}
Finally, for all $x,y\in\mathcal{H}$, one has
\begin{equation}\label{fontos5}
g(x)-g(y)=(g_t(x)-g_t(x_t))+(g_t(x_t)-g_t(y))+\frac{\e(t)}{2}(\|y\|^2-\|x\|^2)\le g_t(x)-g_t(x_t)+\frac{\e(t)}{2}\|y\|^2.
\end{equation}

Now we are ready to prove the main result of the present section.

\begin{theorem}\label{StrongConvergenceResultLyapunov}
Consider $0 < q < 1$  and assume that the Tikhonov parametrization function satisfies the following conditions:
\begin{align}
  &\text{There exists } K_1 > 0 \text{ and }r \in (q, 1) \text{ such that } \e(t) \ge \dfrac{K_1}{t^{r+q}} \mbox{ for } t \mbox{ big enough}.\label{C1}\tag{C1} \\
  &\text{If }\g\neq 0 \text{ then } \e(t) \le \dfrac{\a}{2\g t^{q}} \mbox{ for } t \mbox{ big enough.}\label{C2}\tag{C2} \\
  &\text{There exists } K_3 > 0 \text{ such that }  \left( \dfrac{\dot{\e}(t)}{\e(t)} \right)^2 \le \dfrac{K_3}{t^2} \mbox{ for } t \mbox{ big enough}.\label{C3}\tag{C3}
\end{align}
 For some starting points $u_0, \, v_0\in\mathcal{H}$ let $x : [t_0, \infty) \to \mathcal{H}$ be the unique global solution of  \eqref{DynSys}. Then, the following statements are true.

(i) (convergence of the trajectory) $x(t)$ converges strongly, as $t \to+\infty$, to $x^*$, the element of minimum norm of $\argmin g$;

(ii) (fast convergence rates) $ g_t(x(t) + \b(t)\dot{x}(t) )-g_t(x_t)=\mathcal{O}(t^{2r-2q-2}+\e(t)t^{r-1}),$
$g(x(t) + \b(t)\dot{x}(t) )-\min g=\mathcal{O}(t^{2r-2q-2}+\e(t)),$
$\|\dot{x}(t) \|=\mathcal{O}(t^{r-q-1}+\sqrt{\e(t)}t^{\frac{r-1}{2}})$
and
$\|x(t)-x_t\|=\mathcal{O}(t^{r-1}+\sqrt{\e(t)}t^{\frac{2q+r-1}{2}})\mbox{ as }t\to+\infty;$

(iii) (integral estimates) If $t^{2q-1}\e(t)\in L^1[t_0,+\infty[$ then
$\int_{t_0}^{+\infty} t^q\left( g_t(x(t)+\b(t)\dot{x}(t))-g_t(x_t) \right)dt<+\infty$ and
$\int_{t_0}^{+\infty} t^q\|\dot{x}(t)\|^2 dt<+\infty.$ Further,
$\int_{t_0}^{+\infty}t^{2q}\| \n g(x(t) + \b(t) \dot{x}(t)) +\e(t)x(t)\|^2dt<+\infty,\mbox{ if }\g>0$
and
$\int_{t_0}^{+\infty}t^{2q-1}\| \n g(x(t) + \b(t) \dot{x}(t))\|^2dt<+\infty,\mbox{ if }\g=0.$

\end{theorem}

\begin{proof}
Let us define, for every $t \geq t_0$, the following energy functional
\begin{align}\label{Lyapunovstr}
E(t)&=a(t)(g_t(x(t) + \b(t)\dot{x}(t) )-g_t(x_t))+\frac{1}{2} \| b(x(t)-x_t) + c(t) \dot{x}(t) \|^2 + \frac{d(t)}{2}\|x(t)-x_t\|^2,
\end{align}
where $b \in (0, \a)$, $c(t) = t^q$, $d(t) = b(\a - b - qt^{q-1})$ and
$
a(t) = \dfrac{c^2(t) - bc(t)\b(t)}{1 + \b^\prime(t) - \dfrac{\a\b(t)}{t^q} + \b^2(t)\e(t)}.
$
Obviously there exists $t_1\ge t_0$ such that $a(t)\ge 0$ and $d(t)\ge 0$ for all $t\ge t_1$. Observe further that $a(t)=\mathcal{O}(t^{2q})$ as  $t\to+\infty.$ Therefore, by deriving $E(t)$ we get
\begin{align}\label{DerivLyapunovStr1}
\dot{E}(t)&=a^\prime(t)(g_t(x(t)+\b(t)\dot{x}(t))-g_t(x_t)) + a(t) \left( \frac{d}{dt} g_t( x(t) + \b(t) \dot{x}(t) ) - \frac{d}{dt} g_t(x_t) \right) \\
\nonumber&+ \left\< b(x(t)-x_t) + c(t)\dot{x}(t), b\left( \dot{x}(t) - \frac{d}{dt} x_t \right) + c^\prime(t)\dot{x}(t) + c(t)\ddot{x}(t) \right\> \\
\nonumber& + \frac{d^\prime(t)}{2} \| x(t) - x_t \|^2 + d(t) \left\< x(t) - x_t, \dot{x}(t) - \frac{d}{dt} x_t \right\> \, , \text{ for all } t \ge t_1.
\end{align}
In one hand, from \eqref{fontos0} and \eqref{fontos4}, we find that
\begin{align}\label{DerivGradient0}
\frac{d}{dt} g_t(x_t) = \left\< \n g_t(x_t), \frac{d}{dt} x_t \right\> + \frac{\dot{\e}(t)}{2} \| x_t \|^2 = \frac{\dot{\e}(t)}{2} \| x_t \|^2.
\end{align}
On the other hand, using  \eqref{DynSys} along with \eqref{fontos4} and \eqref{fontos0general} it follows that
\begin{align}\label{DerivGradient1}
\frac{d}{dt} g_t(x(t) + \b(t) \dot{x}(t)) &= \frac{\dot{\e}(t)}{2} \| x(t) + \b(t) \dot{x}(t) \|^2 +  \b(t)\e(t) \left( 1 + \b^\prime(t) - \frac{\a\b(t)}{t^q} \right) \| \dot{x}(t) \|^2\\
\nonumber&- \b(t)\e^2(t) \| x(t) \|^2 - \b(t) \| \n g(x(t) + \b(t) \dot{x}(t)) \|^2\\
\nonumber&+\left( 1 + \b^\prime(t) - \frac{\a\b(t)}{t^q} - \b^2(t)\e(t) \right) \<  \n g(x(t) + \b(t) \dot{x}(t)),\dot{x}(t) \> \\
\nonumber&+ \left( \e(t) \left( 1 + \b^\prime(t) - \frac{\a\b(t)}{t^q} \right) - \b^2(t)\e^2(t) \right) \< x(t), \dot{x}(t) \>  \\
\nonumber&- 2\b(t)\e(t) \<  \n g(x(t) + \b(t) \dot{x}(t)),x(t) \> .
\end{align}
Further, by using the dynamical system \eqref{DynSys}, we obtain
\begin{align}\label{DerivGradient2}
 &\left\< b(x(t)-x_t) + c(t)\dot{x}(t),  b\dot{x}(t) + c^\prime(t)\dot{x}(t) + c(t)\ddot{x}(t) \right\>=c(t) \left( b + c^\prime(t) - \frac{\a c(t)}{t^q} \right) \| \dot{x}(t) \|^2 \\
  \nonumber&+ b \left( b + c^\prime(t) - \frac{\a c(t)}{t^q} \right) \< x(t) - x_t, \dot{x}(t) \> - \e(t)c(t)b \< x(t) - x_t, x(t) \> - \e(t)c^2(t) \< x(t),\dot{x}(t) \>\\
\nonumber&  - b c(t) \<  \n g(x(t) + \b(t) \dot{x}(t)),x(t) - x_t \>  - c^2(t) \<  \n g(x(t) + \b(t) \dot{x}(t)),\dot{x}(t) \>.
\end{align}
Combining \eqref{DerivLyapunovStr1} with \eqref{DerivGradient0}, \eqref{DerivGradient1} and \eqref{DerivGradient2}, and taking into account that $d(t) + b \left( b + c^\prime(t) - \frac{\a c(t)}{t^q}\right) =0$  and $a(t) \left( 1 + \b^\prime(t) - \frac{\a\b(t)}{t^q} - \b^2(t)\e(t) \right)=c^2(t)-bc(t)\b(t)$  it follows that
\begin{align}\label{DerivLyapunovStr2}
\dot{E}(t) &= a^\prime(t) \left( g_t(x(t)+\b(t)\dot{x}(t))-g_t(x_t) \right) + \frac{a(t)\dot{\e}(t)}{2} \left( \| x(t) + \b(t) \dot{x}(t) \|^2 - \| x_t \|^2 \right) \\
\nonumber&+ \frac{d^\prime(t)}{2} \| x(t) - x_t \|^2 - a(t)\b(t) \| \n g(x(t) + \b(t) \dot{x}(t)) \|^2 - a(t)\b(t)\e^2(t) \| x(t) \|^2 \\
\nonumber&+ \left( a(t)\b(t)\e(t) \left( 1 + \b^\prime(t) - \frac{\a\b(t)}{t^q} \right) + c(t) \left( b + c^\prime(t) - \frac{\a c(t)}{t^q} \right) \right) \| \dot{x}(t) \|^2 \\
\nonumber&-bc(t)\b(t)\e(t) \<  x(t),\dot{x}(t) \> - b\e(t)c(t) \< x(t) - x_t, x(t) \>  -bc(t)\b(t)  \<  \n g(x(t) + \b(t) \dot{x}(t)) ,\dot{x}(t)\> \\
\nonumber&- 2 \b(t)a(t)\e(t) \<  \n g(x(t) + \b(t) \dot{x}(t)) , x(t)\> - bc(t) \<  \n g(x(t) + \b(t) \dot{x}(t)),x(t) - x_t \> \\
\nonumber&- bc(t) \left\< \dot{x}(t), \frac{d}{dt} x_t \right\> - \left( b^2 + d(t) \right) \left\< x(t) - x_t, \frac{d}{dt} x_t \right\> \, , \text{ for all } t \ge t_1.
\end{align}

Observe that by using  \eqref{fontos2} we get
\begin{align*}
&-bc(t)\b(t)\e(t) \<  x(t),\dot{x}(t) \> - b\e(t)c(t) \< x(t) - x_t, x(t) \>  -bc(t)\b(t)  \<  \n g(x(t) + \b(t) \dot{x}(t)) ,\dot{x}(t)\> \\
& - bc(t) \<  \n g(x(t) + \b(t) \dot{x}(t)),x(t) - x_t \> =- bc(t) \<  \n g_t(x(t) + \b(t) \dot{x}(t)),x(t)+\b(t) \dot{x}(t) - x_t \> \\
&+b\b(t)\e(t)c(t)\<\dot{x}(t),x(t)+\b(t) \dot{x}(t) - x_t \>\le bc(t) \left( g_t(x_t) - g_t(x(t) + \b(t) \dot{x}(t)) \right)\\
&- \frac{bc(t)\e(t)}{2} \| x(t) + \b(t) \dot{x}(t) - x_t \|^2+b\b(t)\e(t)c(t)\<\dot{x}(t),x(t)+\b(t) \dot{x}(t) - x_t \>=\\
&bc(t) \left( g_t(x_t) - g_t(x(t) + \b(t) \dot{x}(t)) \right)- \frac{bc(t)\e(t)}{2} \| x(t)  - x_t \|^2+ \frac{bc(t)\e(t)\b^2(t)}{2} \| \dot{x}(t)\|^2.
\end{align*}

Further, we have
\begin{align*}
&-a(t)\b(t) \| \n g(x(t) + \b(t) \dot{x}(t)) \|^2 -a(t)\b(t)\e^2(t) \| x(t) \|^2-2\b(t)a(t)\e(t) \<  \n g(x(t) + \b(t) \dot{x}(t)),x(t) \> = \\
&-a(t)\b(t)\|\n g(x(t) + \b(t) \dot{x}(t))+\e(t)x(t) \|^2
\end{align*}
Consequently, \eqref{DerivLyapunovStr2} becomes

\begin{align}\label{DerivLyapunovStr3}
\dot{E}(t) &= (a^\prime(t)-bc(t)) \left( g_t(x(t)+\b(t)\dot{x}(t))-g_t(x_t) \right) + \frac{a(t)\dot{\e}(t)}{2} \left( \| x(t) + \b(t) \dot{x}(t) \|^2 - \| x_t \|^2 \right) \\
\nonumber&+ \frac{d^\prime(t)-bc(t)\e(t)}{2} \| x(t) - x_t \|^2-a(t)\b(t)\|\n g(x(t) + \b(t) \dot{x}(t))+\e(t)x(t) \|^2 \\
\nonumber&+ \left( a(t)\b(t)\e(t) \left( 1 + \b^\prime(t) - \frac{\a\b(t)}{t^q} \right) + c(t) \left( b + c^\prime(t) - \frac{\a c(t)}{t^q} \right)+\frac{bc(t)\e(t)\b^2(t)}{2} \right) \| \dot{x}(t) \|^2 \\
\nonumber&- bc(t) \left\< \dot{x}(t), \frac{d}{dt} x_t \right\> - \left( b^2 + d(t) \right) \left\< x(t) - x_t, \frac{d}{dt} x_t \right\> \, , \text{ for all } t \ge t_1.
\end{align}

Now we use \eqref{fontos1} for evaluating the term involving $\left\< \frac{d}{dt} x_t, \dot{x}(t) \right\>$ and we conclude that for every function $n(t) > 0$ one has
\begin{align}\label{DerivGradient4}
-bc(t)\left\< \frac{d}{dt} x_t, \dot{x}(t) \right\> &\le  \frac{bc(t)}{2} \left( \frac{1}{n(t)} \left( \frac{\dot{\e}(t)}{\e(t)} \right)^2 \| x_t \|^2 + n(t)\| \dot{x}(t) \|^2 \right).
\end{align}
Similarly, for any function $m(t) > 0$ we obtain that
\begin{align}\label{DerivGradient5}
- \left( b^2 + d(t) \right) \left\< \frac{d}{dt} x_t, x(t) - x_t \right\> &\le \frac{b^2+ d(t)}{2} \left( m(t) \left( \frac{\dot{\e}(t)}{\e(t)} \right)^2 \left\| x_t \right\|^2 + \frac{1}{m(t)} \| x(t) - x_t \|^2 \right).
\end{align}

Hence,  by using \eqref{DerivGradient4} and \eqref{DerivGradient5} we get

\begin{align}\label{DerivLyapunovStr4}
\dot{E}(t) &\le \left( a^\prime(t) - bc(t) \right) \left( g_t(x(t)+\b(t)\dot{x}(t))-g_t(x_t) \right) + \frac{a(t)\dot{\e}(t)}{2}  \| x(t) + \b(t) \dot{x}(t) \|^2  \\
\nonumber&+ \left(\frac{d^\prime(t)}{2} - \frac{bc(t)\e(t)}{2}+\frac{b^2+ d(t)}{2m(t)}\right) \| x(t) - x_t \|^2 - a(t)\b(t) \| \n g(x(t) + \b(t) \dot{x}(t)) +\e(t)x(t)\|^2\\
\nonumber&+ \left( a(t)\b(t)\e(t) \left( 1 + \b^\prime(t) - \frac{\a\b(t)}{t^q} \right) + c(t) \left( b+ c^\prime(t) - \frac{\a c(t)}{t^q} \right) + \frac{bc(t)\b^2(t)\e(t)}{2}+\frac{bn(t) c(t)}{2} \right) \| \dot{x}(t) \|^2 \\
\nonumber& +\left(\left(\frac{bc(t)}{2n(t)}+ \frac{b^2+ d(t)}{2}m(t)\right)\left( \frac{\dot{\e}(t)}{\e(t)} \right)^2 -\frac{a(t)\dot{\e}(t)}{2} \right) \| x_t \|^2\, , \text{ for all } t \ge t_1.
\end{align}

Consider now  $K> 0$ which shall be defined later and $q<r< 1$ defined at condition (C1). Knowing that
\begin{align*}
\frac{1}{2} \| b(x(t) - x_t) + c(t)\dot{x}(t) \|^2 \le b^2 \| x(t) - x_t \|^2 + c^2(t)  \| \dot{x}(t) \|^2,
\end{align*}
we get
$$\frac{K}{t^r}E(t)\le \frac{K a(t)}{t^r} \left( g_t(x(t)+\b(t)\dot{x}(t))-g_t(x_t) \right)+\frac{K}{t^r} c^2(t)  \| \dot{x}(t) \|^2+\frac{K}{t^r} \left(\frac{d(t)}{2}+b^2\right)\|x(t)-x_t\|^2.$$

Hence, we take $m(t)=\frac{1}{M}t^r$ and $n(t)=Nt^{q-r}$ for some $N,M>0$ and we get
\begin{align}\label{DerivLyapunovStr5}
\dot{E}(t)& + \frac{K}{t^r} E(t) \le \left( a^\prime(t) - bc(t)+\frac{K a(t)}{t^r} \right) \left( g_t(x(t)+\b(t)\dot{x}(t))-g_t(x_t) \right) + \frac{a(t)\dot{\e}(t)}{2}  \| x(t) + \b(t) \dot{x}(t) \|^2  \\
\nonumber&+ \left(\frac{d^\prime(t)}{2} - \frac{bc(t)\e(t)}{2}+\frac{1}{t^r}\left(K \frac{d(t)+2b^2}{2}+M\frac{b^2+ d(t)}{2}\right)\right) \| x(t) - x_t \|^2\\
\nonumber& - a(t)\b(t) \| \n g(x(t) + \b(t) \dot{x}(t)) +\e(t)x(t)\|^2\\
\nonumber&+ \left( a(t)\b(t)\e(t) \left( 1 + \b^\prime(t) - \frac{\a\b(t)}{t^q} \right) + c(t) \left( b+ c^\prime(t) - \frac{\a c(t)}{t^q} \right) + \frac{bc(t)\b^2(t)\e(t)}{2}+\left(K+\frac{bN}{2}\right)t^{2q-r}\right) \| \dot{x}(t) \|^2 \\
\nonumber& +\left( \left(\frac{b^2+ d(t)}{2M}+\frac{b}{2N}\right)t^r\left( \frac{\dot{\e}(t)}{\e(t)} \right)^2 -\frac{a(t)\dot{\e}(t)}{2} \right) \| x_t \|^2\, , \text{ for all } t \ge t_1.
\end{align}

Let us evaluate the coefficients on the right hand side of \eqref{DerivLyapunovStr5}.

  Since $$a(t)=t^{2q}+\frac{(\a-b)\g t^q+(\a-b)\b +\b q t^{q-1}-\left(\g+\frac{\b}{t^q}\right)^2\e(t)t^{2q}}{1-\frac{\b q}{t^{q+1}}-\frac{\a\g}{t^q}-\frac{\a\b}{t^{2q}}+\left(\g+\frac{\b}{t^q}\right)^2\e(t)},$$
further $q<1$ and $r>q$ we conclude that for $t$ big enough one has
$$a^\prime(t) - bc(t)+\frac{K a(t)}{t^r} <0.$$
Obviously the coefficient of $\| x(t) + \b(t) \dot{x}(t) \|^2$ and $\| \n g(x(t) + \b(t) \dot{x}(t)) +\e(t)x(t)\|^2$ are nonpositive for all $t\ge t_1.$

Concerning the coefficient of $ \| x(t) - x_t \|^2$ one has
\begin{align*}
\frac{d^\prime(t)}{2} - \frac{bc(t)\e(t)}{2}+\frac{1}{t^r}\left(K \frac{d(t)+2b^2}{2}+M\frac{b^2+ d(t)}{2}\right)=&\frac{b}{2}\left(-t^q\e(t)+(K(\a+b)+M\a)t^{-r}-q(K+M)t^{q-r-1}\right)\\
&+\frac{bq(1-q)}{2}t^{q-2}.
\end{align*}
Now, according to (C1) there exists $K_1>0$ such that $\e(t)\ge\frac{K_1}{t^{q+r}}$ for $t$ big enough, hence taking into account that $b<\a$ we take $K<\frac{K_1}{4\a}$, $M<\frac{K_1}{2\a}$ and we get that
$$\frac{d^\prime(t)}{2} - \frac{bc(t)\e(t)}{2}+\frac{1}{t^r}\left(K \frac{d(t)+2b^2}{2}+M\frac{b^2+ d(t)}{2}\right)\le 0$$
for $t$ big enough.

Finally, by taking into account that $a(t)=t^{2q}+\mathcal{O}(t^q)$ as $t\to+\infty$, concerning the coefficient of $\|\dot{x}(t)\|^2$ we conclude the following. If $\g>0$ then
\begin{align*}
& a(t)\b(t)\e(t) \left( 1 + \b^\prime(t) - \frac{\a\b(t)}{t^q} \right) + c(t) \left( b+ c^\prime(t) - \frac{\a c(t)}{t^q} \right) + \frac{bc(t)\b^2(t)\e(t)}{2}+\left(K+\frac{bN}{2}\right)t^{2q-r}=\\
&(b-\a)t^q+\g\e(t)t^{2q}+\mathcal{O}(t^q\e(t))+\mathcal{O}(t^{2q-r})\mbox{ as }t\to+\infty,
\end{align*}
hence, by taking account condition (C2), i.e., $\e(t)\le \frac{\a}{2\g t^q}$ and by fixing $b<\frac{\a}{2}$ we obtain that
$$a(t)\b(t)\e(t) \left( 1 + \b^\prime(t) - \frac{\a\b(t)}{t^q} \right) + c(t) \left( b+ c^\prime(t) - \frac{\a c(t)}{t^q} \right) + \frac{bc(t)\b^2(t)\e(t)}{2}+\left(K+\frac{bN}{2}\right)t^{2q-r}\le 0$$
for $t$ big enough.

In case $\g=0,\,\b>0$ we have
\begin{align*}
& a(t)\b(t)\e(t) \left( 1 + \b^\prime(t) - \frac{\a\b(t)}{t^q} \right) + c(t) \left( b+ c^\prime(t) - \frac{\a c(t)}{t^q} \right) + \frac{bc(t)\b^2(t)\e(t)}{2}+\left(K+\frac{bN}{2}\right)t^{2q-r}=\\
&(b-\a)t^q+\mathcal{O}(t^{q}\e(t))+\mathcal{O}(t^{2q-r})\mbox{ as }t\to+\infty,
\end{align*}
and by since $b<\frac{a}{2}\a$ and taking into account that $\e(t)$ is nonincreasing we get that also in this case
$$a(t)\b(t)\e(t) \left( 1 + \b^\prime(t) - \frac{\a\b(t)}{t^q} \right) + c(t) \left( b+ c^\prime(t) - \frac{\a c(t)}{t^q} \right) + \frac{bc(t)\b^2(t)\e(t)}{2}+\left(K+\frac{bN}{2}\right)t^{2q-r}\le 0$$
for $t$ big enough.

Consequently, there exists $t_2\ge t_1$ such that for all $t\ge t_2$ one has
\begin{align}\label{DerivLyapunovStr6}
\dot{E}(t)+ \frac{K}{t^r} E(t) \le \left(  \left(\frac{b^2+ d(t)}{2M}+\frac{b}{2N}\right)t^r\left( \frac{\dot{\e}(t)}{\e(t)} \right)^2 -\frac{a(t)\dot{\e}(t)}{2} \right) \| x_t \|^2.
\end{align}
Now, according to (C3) there exists $K_3>0$ such that $\left(\frac{\dot{\e}(t)}{\e(t)} \right)^2\le \frac{K_3}{t^2}$ for $t$ big enough
and therefore
$0\le- \dot{\e}(t) \le\frac{\sqrt{K_3}}{t}\e(t)$ for $t$ big enough.

Further, one has $\|x_t\|\le\|x^*\|$ where $x^*$ is the minimal norm element from the set $\argmin g.$ Hence, by taking into account also the form of $a(t)$ we conclude that there exists $t_3\ge t_2$ and a constant $C_0>0$ such that
\begin{align}\label{fi}
&\left( \left(\frac{b^2+ d(t)}{2M}+\frac{b}{2N}\right)t^r\left( \frac{\dot{\e}(t)}{\e(t)} \right)^2 -\frac{a(t)\dot{\e}(t)}{2} \right) \| x_t \|^2\le C_0(t^{r-2}+t^{2q-1}\e(t))\mbox{ for all }t\ge t_3.
\end{align}
Consequently, \eqref{DerivLyapunovStr6} leads to
\begin{align}\label{DerivLyapunovStr7}
\dot{E}(t)+ \frac{K}{t^r} E(t) \le C_0(t^{r-2}+t^{2q-1}\e(t))\mbox{ for all }t\ge t_3.
\end{align}

By multiplying \eqref{DerivLyapunovStr7} with $e^{\frac{K}{1-r}t^{1-r}}$ we obtain

\begin{align}\label{DerivLyapunovStr8}
\frac{d}{dt}\left(e^{\frac{K}{1-r}t^{1-r}} E(t)\right)\le C_0(t^{r-2}+t^{2q-1}\e(t))e^{\frac{K}{1-r}t^{1-r}}\mbox{ for all }t\ge t_3.
\end{align}
By integrating \eqref{DerivLyapunovStr8} on an interval $[t_3,T],\, T>t_3$ we obtain

\begin{align}\label{DerivLyapunovStr9}
e^{\frac{K}{1-r}T^{1-r}} E(T)-\mathcal{C}_0\le C_0\int_{t_3}^T(t^{r-2}+t^{2q-1}\e(t))e^{\frac{K}{1-r}t^{1-r}}dt,
\end{align}
where $\mathcal{C}_0$ is the constant $e^{\frac{K}{1-r}t_3^{1-r}} E(t_3).$

In what follows we will discuss the right hand side of \eqref{DerivLyapunovStr9}.
First of all, note that for a function $t\rightarrow t^ue^{ct^v},u\in\R,\, c,v>0$ one has
\begin{equation}\label{lfos}
\frac{d}{dt}\left(t^ue^{ct^v}\right)=(ut^{u-1}+cvt^{u+v-1})e^{ct^v}\ge c_1 t^{u+v-1}e^{ct^v},
\end{equation}
for some $c_1>0$ and  $t$ big enough.
Consequently, by deploying \eqref{lfos} we conclude that there exists $t_4\ge t_3$ and $c_1,c_2>0$ such that for all $t\ge t_4$ it holds
$$t^{r-2}e^{\frac{K}{1-r}t^{1-r}}\le c_1\frac{d}{dt}\left(t^{2r-2}e^{\frac{K}{1-r}t^{1-r}}\right)$$
and
$$t^{2q-1}\e(t)e^{\frac{K}{1-r}t^{1-r}}\le c_2\e(t)\frac{d}{dt}\left(t^{2q+r-1}e^{\frac{K}{1-r}t^{1-r}}\right).$$

Therefore, there is a constant $\mathcal{K}$ such that
\begin{equation}\label{lfos1}
\int_{t_3}^T t^{r-2}e^{\frac{K}{1-r}t^{1-r}}dt\le \mathcal{C}+ c_1 \int_{t_4}^T\frac{d}{dt}\left(t^{2r-2}e^{\frac{K}{1-r}t^{1-r}}\right)dt=c_1T^{2r-2}e^{\frac{K}{1-r}T^{1-r}}+\mathcal{K},
\end{equation}
where $\mathcal{C}=\int_{t_3}^{t_4} t^{r-2}e^{\frac{K}{1-r}t^{1-r}}dt$ and $\mathcal{K}=\mathcal{C}-c_1t_4^{2r-2}e^{\frac{K}{1-r}t_4^{1-r}}.$

Similarly, by using  integration by parts and also that $- \dot{\e}(t) \le\frac{\sqrt{K_3}}{t}\e(t)$ for all $t\ge t_4$ we obtain

\begin{align}\label{lfos11}
\int_{t_3}^T \e(t)t^{2q-1}e^{\frac{K}{1-r}t^{1-r}}dt&\le \mathcal{C}_0+ c_2\int_{t_4}^T \e(t)\frac{d}{dt}\left(t^{2q+r-1}e^{\frac{K}{1-r}t^{1-r}}\right)dt=\mathcal{C}_1+c_2 \e(T) T^{2q+r-1}e^{\frac{K}{1-r}T^{1-r}}\\
\nonumber&\,+c_2\int_{t_4}^T -\dot{\e}(t)t^{2q+r-1}e^{\frac{K}{1-r}t^{1-r}}dt\\
\nonumber&\le \mathcal{C}_1+c_2 \e(T) T^{2q+r-1}e^{\frac{K}{1-r}T^{1-r}}+c_2\sqrt{K_3}\int_{t_4}^T \e(t)t^{2q+r-2}e^{\frac{K}{1-r}t^{1-r}}dt,
\end{align}
where $\mathcal{C}_0=\int_{t_3}^{t_4} \e(t)t^{2q-1}e^{\frac{K}{1-r}t^{1-r}}dt$ and $\mathcal{C}_1=\mathcal{C}_0-c_2 \e(t_4) t_4^{2q+r-1}e^{\frac{K}{1-r}t_4^{1-r}}.$

 Now, since $r<1$ there exists $t_5\ge t_4$ such that
 $c_2\sqrt{K_3} \e(t)t^{2q+r-2}e^{\frac{K}{1-r}t^{1-r}}\le \frac12 \e(t)t^{2q-1}e^{\frac{K}{1-r}t^{1-r}}$ for all $t\ge t_5$, hence
\begin{align}\label{llfos}
c_2\sqrt{K_3}\int_{t_4}^T \e(t)t^{2q+r-2}e^{\frac{K}{1-r}t^{1-r}}dt&\le c_2\sqrt{K_3}\int_{t_4}^{t_5} \e(t)t^{2q+r-2}e^{\frac{K}{1-r}t^{1-r}}dt+
\frac12 \int_{t_5}^T \e(t)t^{2q-1}e^{\frac{K}{1-r}t^{1-r}}dt\\
\nonumber&=\mathcal{C}_2-\mathcal{C}_3+\frac12 \int_{t_3}^T \e(t)t^{2q-1}e^{\frac{K}{1-r}t^{1-r}}dt,
\end{align}
where $\mathcal{C}_2=c_2\sqrt{K_3}\int_{t_4}^{t_5} \e(t)t^{2q+r-2}e^{\frac{K}{1-r}t^{1-r}}dt$ and $\mathcal{C}_3=\frac12 \int_{t_3}^{t_5} \e(t)t^{2q-1}e^{\frac{K}{1-r}t^{1-r}}dt.$

Combining \eqref{lfos11} and \eqref{llfos} we obtain that for all $T\ge t_5$ it holds

\begin{align}\label{lfos111}
\int_{t_3}^T \e(t)t^{2q-1}e^{\frac{K}{1-r}t^{1-r}}dt&\le 2(\mathcal{C}_1+\mathcal{C}_2-\mathcal{C}_3)+2c_2 \e(T) T^{2q+r-1}e^{\frac{K}{1-r}T^{1-r}}=\mathcal{C}+2c_2 \e(T) T^{2q+r-1}e^{\frac{K}{1-r}T^{1-r}}.
\end{align}

Hence, for $T$ big enough and for some $C_3>0$, \eqref{lfos1}, \eqref{lfos111} and \eqref{DerivLyapunovStr9} yield to
\begin{align}\label{DerivLyapunovStr91}
 E(T)&\le e^{-\frac{K}{1-r}T^{1-r}}\left(\mathcal{C}_0+C_0\left(\mathcal{K}+c_1T^{2r-2}e^{\frac{K}{1-r}T^{1-r}}+\mathcal{C}+2c_2 \e(T) T^{2q+r-1}e^{\frac{K}{1-r}T^{1-r}}\right)\right)\\
\nonumber&=\frac{\mathcal{K}_0}{e^{\frac{K}{1-r}T^{1-r}}}+C_1T^{2r-2}+C_2\e(T)T^{2q+r-1}\le C_3(T^{2r-2}+\e(T)T^{2q+r-1}),
\end{align}
where $C_1=C_0c_1,\,C_2=2C_0c_2$ and $\mathcal{K}_0=\mathcal{C}_0+C_0(\mathcal{K}+\mathcal{C}).$

 Taking into account the form of $E(t)$ we obtain at once that
$$g_t(x(t) + \b(t)\dot{x}(t) )-g_t(x_t)=\mathcal{O}(t^{2r-2q-2}+\e(t)t^{r-1})\mbox{ as }t\to+\infty,$$
$$\|\dot{x}(t) \|=\mathcal{O}(t^{r-q-1}+\sqrt{\e(t)}t^{\frac{r-1}{2}})\mbox{ as }t\to+\infty$$
and
$$\|x(t)-x_t\|=\mathcal{O}(t^{r-1}+\sqrt{\e(t)}t^{\frac{2q+r-1}{2}})\mbox{ as }t\to+\infty.$$
Now from $r<1,\,q<1$ and (C1) we get $\lim_{t\to+\infty}(t^{r-1}+\sqrt{\e(t)}t^{\frac{2q+r-1}{2}})=0$, further $\lim_{t\to+\infty}x_t=x^*$, hence the latter relation shows that $x(t)$ converges strongly to $x^*$, that is
$$\lim_{t\to+\infty}\|x(t)-x^*\|=0.$$

Even more, in case $\lim_{t\to+\infty}t^{2q}\e(t)>0$  a better estimate can be obtained. Indeed, by deploying \eqref{fontos3} we get

$$g_t(x(t)+\b(t)\dot{x}(t))-g_t(x_t)\ge\frac{\e(t)}{2}\|x(t)+\b(t)\dot{x}(t)-x_t\|^2,$$
hence
$$\|x(t)+\b(t)\dot{x}(t)-x_t\|=\mathcal{O}\left(t^{\frac{r-1}{2}}+\frac{t^{r-q-1}}{\sqrt{\e(t)}}\right)\mbox{ as }t\to+\infty.$$
By using the fact that
$\|\b(t)\dot{x}(t) \|=\mathcal{O}(\b(t)t^{r-q-1}+\b(t)\sqrt{\e(t)}t^{\frac{r-1}{2}})\mbox{ as }t\to+\infty$ we obtain that
$$|x(t)-x_t\|=\mathcal{O}\left(t^{\frac{r-1}{2}}+\frac{t^{r-q-1}}{\sqrt{\e(t)}}\right)\mbox{ as }t\to+\infty.$$

Finally, by using \eqref{fontos5} we get
$$g(x(t)+\b(t)\dot{x}(t))-g(x^*)\le g_t(x(t)+\b(t)\dot{x}(t))-g_t(x_t)+\frac{\e(t)}{2}\|x^*\|^2,$$
hence
$$g(x(t)+\b(t)\dot{x}(t))-\min g=\mathcal{O}(t^{2r-2q-2}+\e(t))\mbox{ as }t\to+\infty.$$

Concerning the integral estimates, let us take $K=0$ in \eqref{DerivLyapunovStr5}. Then,

 \begin{align}\label{DerivLyapunovStr5ie}
\dot{E}(t)& \le \left( a^\prime(t) - bc(t) \right) \left( g_t(x(t)+\b(t)\dot{x}(t))-g_t(x_t) \right) + \frac{a(t)\dot{\e}(t)}{2}  \| x(t) + \b(t) \dot{x}(t) \|^2  \\
\nonumber&+ \left(\frac{d^\prime(t)}{2} - \frac{bc(t)\e(t)}{2}+M\frac{b^2+ d(t)}{2t^r}\right) \| x(t) - x_t \|^2\\
\nonumber& - a(t)\b(t) \| \n g(x(t) + \b(t) \dot{x}(t)) +\e(t)x(t)\|^2\\
\nonumber&+ \left( a(t)\b(t)\e(t) \left( 1 + \b^\prime(t) - \frac{\a\b(t)}{t^q} \right) + c(t) \left( b+ c^\prime(t) - \frac{\a c(t)}{t^q} \right) + \frac{bc(t)\b^2(t)\e(t)}{2}+\frac{bN}{2}t^{2q-r}\right) \| \dot{x}(t) \|^2 \\
\nonumber& +\left( \left(\frac{b^2+ d(t)}{2M}+\frac{b}{2N}\right)t^r\left( \frac{\dot{\e}(t)}{\e(t)} \right)^2 -\frac{a(t)\dot{\e}(t)}{2} \right) \| x_t \|^2\, , \text{ for all } t \ge t_1.
\end{align}

Arguing as before, we conclude that for $t$ big enough, the coefficients of $ \| x(t) + \b(t) \dot{x}(t) \|^2$ and $\| x(t) - x_t \|^2$ are non-positive. Further, there exists $A_1>0$ such that
$$a^\prime(t) - bc(t) \le-A_1 t^q\mbox{ for }t\mbox{ big enough.}$$
Similarly, there exists $A_2>0$ such that for $t$ big enough  $- a(t)\b(t)\le -A_2t^{2q},$  if $\g\neq 0$ and  $- a(t)\b(t)\le -A_2t^{2q-1},$ if $\g=0.$

Finally, if $\g>0$ the coefficient of $\|\dot{x}(t)\|^2$ is
\begin{align*}
& a(t)\b(t)\e(t) \left( 1 + \b^\prime(t) - \frac{\a\b(t)}{t^q} \right) + c(t) \left( b+ c^\prime(t) - \frac{\a c(t)}{t^q} \right) + \frac{bc(t)\b^2(t)\e(t)}{2}+\frac{bN}{2}t^{2q-r}=\\
&(b-\a)t^q+\g\e(t)t^{2q}+\mathcal{O}(t^q\e(t))+\mathcal{O}(t^{2q-r})\mbox{ as }t\to+\infty,
\end{align*}
hence, by taking account condition (C2), i.e., $\e(t)\le \frac{\a}{2\g t^q}$ and by fixing $b<\frac{\a}{2}$ we obtain that there exists $A_3>0$ such  that
$$a(t)\b(t)\e(t) \left( 1 + \b^\prime(t) - \frac{\a\b(t)}{t^q} \right) + c(t) \left( b+ c^\prime(t) - \frac{\a c(t)}{t^q} \right) + \frac{bc(t)\b^2(t)\e(t)}{2}+\frac{bN}{2}t^{2q-r}\le -A_3t^q$$
for $t$ big enough.

In case $\g=0,\,\b>0$ we have
\begin{align*}
& a(t)\b(t)\e(t) \left( 1 + \b^\prime(t) - \frac{\a\b(t)}{t^q} \right) + c(t) \left( b+ c^\prime(t) - \frac{\a c(t)}{t^q} \right) + \frac{bc(t)\b^2(t)\e(t)}{2}+\left(K+\frac{bN}{2}\right)t^{2q-r}=\\
&(b-\a)t^q+\b\e(t)t^{2q-1}+\mathcal{O}(t^{q-1}\e(t))+\mathcal{O}(t^{2q-r})\mbox{ as }t\to+\infty,
\end{align*}
and by fixing $b<\a$ and taking into account that $\e(t)$ is nonincreasing we get that also in this case that there exists $A_3>0$ such that
$$a(t)\b(t)\e(t) \left( 1 + \b^\prime(t) - \frac{\a\b(t)}{t^q} \right) + c(t) \left( b+ c^\prime(t) - \frac{\a c(t)}{t^q} \right) + \frac{bc(t)\b^2(t)\e(t)}{2}+\left(K+\frac{bN}{2}\right)t^{2q-r}\le -A_3t^q$$
for $t$ big enough.

Further, according to \eqref{fi} one has
$\left( \left(\frac{b^2+ d(t)}{2M}+\frac{b}{2N}\right)t^r\left( \frac{\dot{\e}(t)}{\e(t)} \right)^2 -\frac{a(t)\dot{\e}(t)}{2} \right) \| x_t \|^2\le C_0(t^{r-2}+t^{2q-1}\e(t))$ for $t$ big enough, hence \eqref{DerivLyapunovStr5ie} leads to the existence of $t_1'\ge t_1$ such that

\begin{align}\label{DerivLyapunovStr5ie1}
&\dot{E}(t)+A_1t^q\left( g_t(x(t)+\b(t)\dot{x}(t))-g_t(x_t) \right)+A_2(t)\| \n g(x(t) + \b(t) \dot{x}(t)) +\e(t)x(t)\|^2+A_3t^q\|\dot{x}(t)\|^2 \le\\
&\nonumber C_0(t^{r-2}+t^{2q-1}\e(t))\mbox{ for all }t\ge t_1',
\end{align}
where $A_2(t)=A_2 t^{2q}$ if $\g>0$ and $A_2(t)=A_2 t^{2q-1}$ if $\g=0$.
Now, according to the hypotheses of the theorem, we have $t^{2q-1}\e(t)\in L^1[t_0,+\infty[$. Further, since $r<1$ we deduce that
$$ C_0\int_{t_1'}^{+\infty}t^{r-2}+t^{2q-1}\e(t) dt<+\infty.$$
Consequently, by integrating \eqref{DerivLyapunovStr5ie1} on an interval $[t_1', T]$ we obtain that there exists $C_1>0$ such that
\begin{align}\label{DerivLyapunovStr5ie2}
&E(T)+A_1\int_{t_1'}^T t^q\left( g_t(x(t)+\b(t)\dot{x}(t))-g_t(x_t) \right)dt+\int_{t_1'}^T A_2(t)\| \n g(x(t) + \b(t) \dot{x}(t)) +\e(t)x(t)\|^2dt\\
\nonumber&+A_3\int_{t_1'}^Tt^q\|\dot{x}(t)\|^2 dt \le C_1.
\end{align}
Hence,
$$\int_{t_0}^{+\infty} t^q\left( g_t(x(t)+\b(t)\dot{x}(t))-g_t(x_t) \right)dt<+\infty,$$
$$\int_{t_0}^{+\infty} t^q\|\dot{x}(t)\|^2 dt<+\infty,$$
$$\int_{t_0}^{+\infty}t^{2q}\| \n g(x(t) + \b(t) \dot{x}(t)) +\e(t)x(t)\|^2dt<+\infty,\mbox{ if }\g>0$$
and
$$\int_{t_0}^{+\infty}t^{2q-1}\| \n g(x(t) + \b(t) \dot{x}(t)) +\e(t)x(t)\|^2dt<+\infty,\mbox{ if }\g=0.$$

Further, since $t^{2q-1}\e(t)\in L^1[t_0,+\infty[$ obviously $t^{2q-1}\e^2(t)\in L^1[t_0,+\infty[$, hence $$\int_{t_0}^{+\infty}t^{2q-1}\| \n g(x(t) + \b(t) \dot{x}(t))\|^2dt<+\infty,\mbox{ if }\g=0.$$

\end{proof}

\begin{remark} Natural candidates for the Tikhonov regularization parameter are $\e(t)=\frac{a}{t^p},\,a,p>0.$ In this case condition (C1) is nothing else that $p<q+1.$ Indeed, for $p<q+1$ one can  choose $r\in]\max(q, p-q),1[$ and $K_1=a$ and one has $\e(t)\ge\frac{K_1}{t^{r+q}}.$ Note that condition (C2), in which we assume that $\g>0$, holds whenever $p>q$ or $p=q$ and $a\le \dfrac{\a}{2\g}.$ Finally, (C3) always holds.
\end{remark}
Consequently the following result is valid.
\begin{corollary}\label{specstr} Consider $0<q<1$ and $\e(t)=\frac{a}{t^p},\,a>0$ and $q\le p<q+1.$ Further, when $q=p$ assume that $a\le \dfrac{\a}{2\g}.$ Let $x$ be a trajectory generated by the dynamical system \eqref{DynSys}. Then the following statements hold true.

(i)  $x(t)$ converges strongly, as $t \to+\infty$, to $x^*$, the element of minimum norm of $\argmin g$;

(ii)
$g(x(t) + \b(t)\dot{x}(t) )-\min g=\mathcal{O}(t^{-\min(2q,p)})$
and $\|\dot{x}(t) \|=\mathcal{O}(t^{-\frac{\min(2q,p)}{2}})\mbox{ as }t\to+\infty,$
further for every $r\in]\max(q, p-q),1[$ one has
$\|x(t)-x_t\|=\mathcal{O}(t^{r-1}+t^{\frac{2q+r-p-1}{2}})\mbox{ as }t\to+\infty;$

(iii) If $p>2q$ then
$\int_{t_0}^{+\infty} t^q\left( g_t(x(t)+\b(t)\dot{x}(t))-g_t(x_t) \right)dt<+\infty$ and
$\int_{t_0}^{+\infty} t^q\|\dot{x}(t)\|^2 dt<+\infty.$ Further,
$\int_{t_0}^{+\infty}t^{2q}\| \n g(x(t) + \b(t) \dot{x}(t)) +t^{-p}x(t)\|^2dt<+\infty,\mbox{ if }\g>0$
and
$\int_{t_0}^{+\infty}t^{2q-1}\| \n g(x(t) + \b(t) \dot{x}(t))\|^2dt<+\infty,\mbox{ if }\g=0.$

(iv) Moreover, if $p>\max\left(\frac{2q+1}{2},2q\right),\,\g>0$ then
$\int_{t_0}^{+\infty}t^{2q}\| \n g(x(t) + \b(t) \dot{x}(t))\|^2dt<+\infty.$
\end{corollary}
\begin{proof} Note that (i) and (iii) are direct applications of Theorem \ref{StrongConvergenceResultLyapunov}, hence let us show (ii). According to Theorem \ref{StrongConvergenceResultLyapunov} the fast convergence rates in this special case are
$g(x(t) + \b(t)\dot{x}(t) )-\min g=\mathcal{O}(t^{2r-2q-2}+t^{-p}),$
$\|\dot{x}(t) \|=\mathcal{O}(t^{r-q-1}+t^{\frac{r-p-1}{2}})$
and
$\|x(t)-x_t\|=\mathcal{O}(t^{r-1}+t^{\frac{2q+r-p-1}{2}})\mbox{ as }t\to+\infty.$

Now, taking into account that $r<1$ we get $t^{2r-2q-2}+t^{-p}\le t^{-2q}+t^{-p}\le 2 t^{-\min(2q,p)}$ the claim follows.

Concerning (iv), observe that in this case
$$t^{2q}\| \n g(x(t) + \b(t) \dot{x}(t))\|^2\le 2t^{2q}\| \n g(x(t) + \b(t) \dot{x}(t)) +t^{-p}x(t)\|^2+2t^{2q-2p}\|x(t)\|^2$$
and $t^{2q-2p}\|x(t)\|^2\in L^1[t_0,+\infty[,$ hence the claim is a consequence of (iii).
\end{proof}

\section{Conclusion, perspectives}
In this paper we studied a second order dynamical system with a Tikhonov regularization term and an implicit Hessian driven damping in connection to a minimization problem of a smooth convex function. Natural explicit discretization of our dynamical system leads to inertial algorithm of gradient type with a Tikhonov regularization term, hence the results of the present paper opens the gate for study the strong convergence properties of these algorithms. We emphasize that the literature is very poor yet in such results, see \cite{MKSL,Ljems}. Depending on the setting of the parameters involved we obtained both weak and strong convergence of the generated trajectories to a minimizer of the objective function and also fast convergence rates and integral estimates involving the function values in a generated trajectory, velocity and gradient. Our results not only extend but also improve the results from the literature concerning second order dynamical systems with a Tikhonov regularization term.

\section{Statements and Declarations}
{\bf Availability of data and materials}

In this manuscript only the datasets generated by authors were analysed.

{\bf Funding}

The research leading to these results received funding from Romanian Ministry of Research, Innovation and Digitization, CNCS-UEFISCDI, project number PN-III-P1-1.1-TE-2021-0138, within PNCDI III .

{\bf Competing interests}

 The author has no relevant financial or non-financial interests to disclose.

\section{Appendix}
\begin{lemma}\label{nullimit}[Lemma A.3 from \cite{ACR}] Let $\d>0$ and $f\in L^1((\d,+\infty),\R)$ be a nonnegative and continuous function.  Let $\varphi:[\d,+\infty)\To[0,+\infty)$ be a nondecreasing function such that $\lim\limits_{t\to+\infty}\varphi(t)=+\infty.$ Then it holds
$$\lim\limits_{t\to+\infty}\frac{1}{\varphi(t)}\int_\d^t\varphi(s)f(s)ds=0.$$
\end{lemma}

\begin{lemma}\label{fejer-cont1}[Lemma 5.1 from \cite{abbas-att-sv}] Suppose that $F:[t_0,+\infty)\To\R$ is locally absolutely continuous and bounded from below and that
there exists $G\in L^1([t_0,+\infty),\R)$ such that
$$\frac{d}{dt}F(t)\leq G(t)$$
for almost every $t \in [t_0,+\infty)$. Then there exists $\lim\limits_{t\to +\infty} F(t)\in\R$.
\end{lemma}

The continuous version of  the Opial Lemma (see \cite{att-c-p-r-math-pr2018}) is the main tool for proving weak convergence for the generated trajectory.

\begin{lemma}\label{Opial}
Let $S \subseteq \mathcal{H}$ be  a nonempty set and $x : [t_0, +\infty[ \to H$ a given map such that:
\begin{align*}
& (i) \quad \mbox{for every }z \in S \ \mbox{the limit} \ \lim\limits_{t \To +\infty} \| x(t) - z \|  \ \mbox{exists};\\
& (ii) \quad \text{every weak sequential limit point of } x(t) \text{ belongs to the set }S.
\end{align*}
Then the trajectory $x(t)$ converges weakly to an element in $S$ as $t \to + \infty$.
\end{lemma}

The following lemma from \cite{L-jde} is essential for proving weak convergence.
\begin{lemma}\label{forwconv}(Lemma 16 from \cite{L-jde}) Let $t_0>0, \,\a>0,\, 0 <q<1$ and let $w:[t_0, +\infty[\to\R$ be a continuously differentiable function which is bounded from below. Consider $k:[t_0, +\infty[\to\R$ a nonnegative function and assume that the function $\int_{t}^{+\infty}\frac{dT}{ Q(T)}Q(t)k(t)\in L^1[t_0, +\infty[$, where $Q(T) :=\exp\left(\int_{t_0}^T \frac{\a}{s^q}ds\right).$ Assume further, that $$\ddot{w}(t)+\frac{\a}{t^q}\dot{w}(t)\le k(t)$$ for almost every $t>t_0.$ Then, the positive part $[\dot{w}]_+$ of $\dot{w}$ belongs to $L^1[t_0, +\infty[,$ and $lim_{t\to+\infty}w(t)$ exists.
\end{lemma}

\end{document}